\documentclass[amssymb,10pt]{article}
\usepackage[pagewise]{lineno}
\usepackage{graphicx}
\usepackage{indentfirst}
\usepackage{amsmath}
\usepackage{color}

\usepackage{xcolor}
\usepackage{amssymb}
\usepackage{amsthm}
\usepackage{titletoc}
\usepackage{amssymb}
\usepackage{amsxtra}
\usepackage{mathrsfs}
\usepackage{amsmath,amstext,amsfonts,verbatim, url}
\usepackage{epstopdf}
\usepackage{subfig}
\usepackage{psfrag}
\usepackage{enumerate}

\usepackage{amsmath}   
\usepackage{amssymb}   
\usepackage{mathrsfs}
\usepackage{amsmath, amssymb, mathrsfs}

\usepackage{tikz}

\usepackage{xy}

\usepackage[colorlinks=true]{hyperref}
\usepackage{amsfonts}
\usepackage{mathrsfs}
\usepackage{amsthm,amsmath,amssymb,mathrsfs,amsfonts,graphicx,graphics,latexsym,exscale,cmmib57,dsfont,bbm,amscd,ulem}
\usepackage{color,soul}
\newtheorem{thm}{Theorem}[section]
\newtheorem{lem}[thm]{Lemma}
\newtheorem{pro}[thm]{Proposition}

\newtheorem{cor}[thm]{Corollary}

\newtheorem{bsc}[thm]{Fact}

\def\cal#1{\fam2#1}

\def\T{{\mathbb T}}

\def\bb{\begin}
\def\be{\begin{equation}}
	\def\ee{\end{equation}}
\def\bea{\begin{eqnarray}}
	\def\eea{\end{eqnarray}}
\def\beaa{\begin{eqnarray*}}
	\def\eeaa{\end{eqnarray*}}

\def\ifl{\iffalse}

\def\bb{\begin}
           \def\ea{\end{array}}
          \def\ec{\end{center}}
     \def\ed{\end{description}}
\def\be{\bb{equation}}        \def\ee{\end{equation}}
\def\bea{\bb{eqnarray}}       \def\eea{\end{eqnarray}}
\def\beaa{\bb{eqnarray*}}     \def\eeaa{\end{eqnarray*}}
 \def\et{\end{thebibliography}}

       \def\F{{\cal F}}    
\def\V{{\cal V}}   \def\U{{\cal U}}

\def\supp{{\rm supp}}

\def\Leb{{\rm Leb}}

\begin{document}

\title{Variation of Physical Measures in Nontrivial Mixed Partially Hyperbolic Systems}

\date{}
\maketitle

{\center
Hangyue Zhang

}

\vspace{1em}

\begin{abstract}
We construct a \( C^\infty \) nontrivial mixed partially hyperbolic system and explicitly identify its skeleton. This example shares characteristics with the classical examples studied by Smale~\cite{Smale1967}, Shub~\cite{Shub1971}, Kan~\cite{Kan1994}, and Dolgopyat–Viana–Yang~\cite{DolgopyatVianaYang2016}. Moreover, the support of each physical measure contains three fixed points with mutually distinct unstable indices.
By appropriately perturbing the skeleton, we provide an example where the number of physical measures varies upper semicontinuously. The general framework of mixed partially hyperbolic systems has been studied in \cite{MiCaoYang2017, MiCao2021}.
\end{abstract}

\vspace{1em}


\vspace{0.3em}


\section{Introduction}

Let $f : M \to M$ be a $C^{1+}$-diffeomorphism on a smooth Riemannian manifold $M$, and let  \( \mu \) be an \( f \)-invariant measure.  We say that \( \mu \) is a \textbf{physical measure}  if the set
\[
B(\mu) := \left\{ x \in M : \lim_{n \to +\infty} \frac{1}{n} \sum_{j=0}^{n-1} \varphi(f^j(x)) = \int \varphi \, d\mu, \ \forall \varphi \in C^0(M) \right\}
\]
has positive Lebesgue measure. The set \(B(\mu)\) is called the {\textbf{basin} of \(\mu\).  
We write \( E \oplus_{\succ} F \) if
\begin{itemize}
    \item The dimensions of \( E(x) \) and \( F(x) \) are constant over \( M \).
    \item \( Df(E) = E, Df(F) = F \), and \( E \cap F \) is a trivial bundle.
    \item There exist constants \( c > 0 \) and \( \sigma < 1 \) such that, for all nonzero vectors \( v_E \in E \) and \( v_F \in F \), and for all \( n \geq 1 \),
    \[
    \frac{\| Df^n(v_F) \|}{\| Df^n(v_E) \|} \leq c \sigma^n \frac{\| v_F \|}{\| v_E \|}.
    \]
\end{itemize}
We mention that the symbol "$\succ$" in the notation $E \oplus_{\succ} F$ not only signifies that $E$ dominates $F$, but also indicates that the angles between $E$ and $F$ are uniformly bounded away from 0.
Now we say that $f$ is {\textbf{partially hyperbolic} if there exists a  $Df$-invariant  continuous splitting of the tangent bundle
$$
TM=E^{uu}\oplus_\succ E^{cu} \oplus_\succ E^{cs}
$$
such that \( E^{uu} \) is uniformly expanding. 
It is a classical result from \cite{HPS} that the strong unstable bundle \(E^{uu}\) integrates uniquely to an \(f\)-invariant strong unstable foliation \( \mathscr{F}^{uu}(f) \).

 We call $\mu$ a {\textbf{Gibbs $u$-state}} of $f$ if  conditional measures of \( \mu \) on strong unstable leaves are absolutely continuous with respect to Lebesgue measure.  The concept of a Gibbs \( u \)-state was  first introduced by Pesin and Sinai \cite{PS} in 1982.
 In 1984, Ledrappier \cite{Ledrappier1984} proved  that $\mu$ a Gibbs $u$-state} of $f$  if and only if it satisfies following partial entropy formula ~\ref{LS}.
\begin{equation}\label{LS}
h_\mu(f, \mathscr{F}^{uu}(f)) = \int \log |\det Df|_{E^{uu}}|  d\mu,
\end{equation}
where \( h_\mu(f, \mathscr{F}^{uu}(f)) = h_\mu(f, \xi) \) for any measurable partition \( \xi \) that is \( u \)-subordinate to \( \mathscr{F}^{uu}(f) \) with respect to \( \mu \)
 (See \cite[Lemma 3.1.2]{LedrappierYoung1985} and \cite[Lemma 3.2]{Yang2021} for more details about \( h_\mu(f, \xi) \).) 

We say that an invariant subbundle  \(E\) is \textbf{mostly contracting} (resp.  \textbf{mostly expanding} ) if every Gibbs \(u\)-state of \(f\) has only negative (respectively, only positive) Lyapunov exponents along \(E\).

In cases where \( E^{cs} \) is uniformly contracting, \( f \) is called \textbf{ non-uniformly expanding along \( E^{cu} \)} if there exists a positive Lebesgue measure set \( H \) such that for every \( x \in H \),  
\begin{equation}\label{non-unifor}
\limsup_{n \to +\infty} \frac{1}{n}\sum_{j=1}^n\log\|Df^{-1}|_{E^{cu}(f^j(x))}\| < 0.
\end{equation}
This class of partially hyperbolic systems, which includes the case where \( E^{uu} \) is trivial, was investigated by Alves, Bonatti, and Viana in   \cite{AlvesBonattiViana2000} in 2000. They established the existence of physical measures for such systems.

In 2000, Bonatti and Viana \cite{BonattiViana2000} discovered that when \( E^{cu} \) is a trivial bundle, partially hyperbolic diffeomorphisms admit finitely many physical measures, provided that \( E^{cs} \) is mostly contracting.  In 2016, Dolgopyat, Viana, and Yang \cite{DolgopyatVianaYang2016} discovered a skeleton structure that describes the supports and basins of physical measures in such systems. They analyzed the behavior of this skeleton under perturbations. Their results showed that the number of physical measures in these systems exhibits upper semicontinuity. In simpler terms, perturbations cannot cause the number of physical measures to increase compared to the original system. To illustrate this phenomenon, they provided examples of New–Kan–type skew products on \( \mathbb{T}^2 \times S^2 \).



In 2017, Mi, Cao, and Yang \cite{MiCaoYang2017} studied systems where \(E^{cu}\) is mostly expanding and \(E^{cs}\) is mostly contracting. It is worth mentioning that this mixed setting is a \(C^1\)-robust property, as stated in \cite[Theorem~B]{Yang2021}. They showed that such systems admit finitely many physical measures, whose basins together cover a subset of \(M\) of full Lebesgue measure.
This mixed setting generalizes other frameworks for studying physical measures in partially hyperbolic splittings, such as the classical mostly expanding and mostly contracting settings \cite{BonattiViana2000, AnderssonVasquez2018}. Their results partially resolved the problem posed by Alves, Bonatti, and Viana in \cite[Section~6]{AlvesBonattiViana2000}.  Later, in 2020, Mi and Cao \cite{MiCao2021} studied the variation of physical measures in this mixed setting. They obtained the same type of result as Dolgopyat, Viana, and Yang, namely, the upper semicontinuity of the number of physical measures.  Since Dolgopyat, Viana, and Yang provided examples showing how physical measures collapse, it is therefore natural to ask whether there exist mixed examples that also exhibit the collapse of physical measures, that is, examples satisfying the following properties.
\begin{enumerate}
\item \label{first}The map \( f \) admits the partially hyperbolic splitting \( TM = E^{uu} \oplus_\succ E^{cu} \oplus_\succ E^{cs} \), such that \( E^{cs} \) is mostly contracting and \( E^{cu} \) is mostly expanding.

\item \label{second}
For any physical measure \( \mu \) of \( f \), there exist two fixed points \( p \) and \( q \) such that
\begin{equation}\label{liangdian}
 \| Df|_{E^{cu}(p)} \| < 1 \quad \text{and} \quad \| Df^{-1}|_{E^{cs}_{sub}(q)} \| < 1, \quad \{p, q\} \subset \supp(\mu),
\end{equation}
where \( E^{cs}_{sub}(q) \) is a proper invariant subbundle of \( E^{cs}(q) \).

\item \label{third} The number of physical measures varies in an upper semi-continuous manner.
\end{enumerate}
(We briefly explain property~\eqref{second}: If \( \| Df^{-1}|_{E^{cs}(q)} \| < 1 \), then, by the dominated splitting, \( q \) would be a repelling point. This means that \( Df^{-1} \) contracts vectors in the subbundle \( E^{cs}(q) \), leading to \( q \) being a repelling fixed point. As a result, \( q \) could not lie in the support of any non-atomic measure. Moreover, by the dominated splitting, we conclude that \( p \neq q \), since the two points must lie in distinct regions of the dynamics.)
Motivated by this question, we undertook the present work and provided an affirmative answer. This may help to fill the gap in \cite{MiCao2021}.

\bigskip

\begin{thm}[Main Theorem]\label{main}
There exists a smooth diffeomorphism \( f \) on \( \mathbb{T}^2 \times \mathbb{S}  \times \mathbb{T}^2 \) that admits two physical measures and satisfies properties \eqref{first}, \eqref{second} and \eqref{third} where \(\mathbb{S} = \mathbb{R}/2\mathbb{Z}\).
\end{thm}

\bigskip

We remark that, in our construction, any even number of physical measures can be produced while still satisfying properties properties \eqref{first}, \eqref{second} and \eqref{third}. Indeed, by taking
\[
\mathbb{S} = \mathbb{R}/2k\mathbb{Z}, \qquad k \in \mathbb{N}^+
\]
the resulting diffeomorphism admits \(2k\) distinct physical measures. For simplicity, we focus on the case of two physical measures. Under small perturbations, these two measures collapse into a single one. This simplified case is sufficient to demonstrate the phenomenon of interest.

The reason our construction embodies certain characteristics of many classical examples (e.g., Smale~\cite{Smale1967}, Shub~\cite{Shub1971}, Kan~\cite{Kan1994}, and Dolgopyat–Viana–Yang~\cite{DolgopyatVianaYang2016}),  stems from the following Proposition~\ref{laiyuan2}, which is inspired by \cite{GanLiVianaYang2021}. 
\begin{pro}\label{laiyuan2}
 Consider the skew  product diffeomorphism defined by
\[
f(x, y) = (\widetilde{A}(x), K(x, y)) \quad \text{on} \quad \mathbb{T}^2 \times N,
\]
where \( N \) is a compact smooth Riemannian manifold and \( \widetilde{A} \) is a diffeomorphism on \( \mathbb{T}^2 \). Then  there does not exist a physical measure for which all the Lyapunov exponents along the invariant subbundle \( TN \) are positive.
\end{pro}
Since this is the beginning of the source of this paper, we provide a direct proof.
\begin{proof}
The proof proceeds by contradiction.
 Suppose such a physical measure \( \mu \) exists. 
 Then, by the Oseledets multiplicative ergodic theorem (see \cite{Oseledec1968} and \cite[Theorem~4.2]{Viana2014}), we obtain
\[
\int \log |\det \left( Df \big|_{TN} \right)| \, d\mu > 0.
\]
Since \( (\Leb_{\mathbb{T}^2} \times \Leb_{N})( B(\mu)) > 0\), it follows from Fubini's theorem (\(\Leb_{\mathbb{T}^2} \times \Leb_{N}\)) that there would exist a \( \Leb_{N} \)-positive measurable set 
\[
\Lambda_{N}(a, x_0) \subset \{ x_0 \} \times N
\]
such that for every \( (x_0, y) \in \Lambda_{N}(a, x_0) \), we have
\begin{align*}
\lim_{n \to +\infty} \frac{1}{n} \sum_{0 \le i \le n-1} \log \left| \det \left( Df \big|_{T N(f^i(x_0, y))} \right) \right| 
&= \lim_{n \to +\infty} \frac{1}{n} \log \left| \det \left( Df^n \big|_{T N(x_0, y)} \right) \right| \\
&\geq a > 0.
\end{align*}
For convenience, denote by \( \Leb_{N} \) the Lebesgue measure on \( \{ x \} \times N \) for any \( x \in \mathbb{T}^2 \).
Then
\[
\Lambda_{N}(a,x_0) \subset \bigcup_{i \in \mathbb{N}} \bigcap_{n \geq i} \left\{(x_0, y): \frac{1}{n} \log \left| \det \left( Df^n \big|_{TN(x_0,y)} \right) \right| \geq \frac{a}{2} \right\}.
\]
There exists \(n_0\) such that 
\[
\Leb_{N} \left( \bigcap_{n \geq n_0} \left\{(x_0, y): \frac{1}{n} \log \left| \det \left( Df^n \big|_{TN(x_0,y)} \right) \right| \geq \frac{a}{2} \right\} \right) > 0.
\]
Denote
\[
\Lambda=\bigcap_{n \geq n_0} \left\{(x_0, y): \frac{1}{n} \log \left| \det \left( Df^n \big|_{TN(x_0,y)} \right) \right| \geq \frac{a}{2} \right\}.
\]
 Since \(f \) permutes the set \( \{ \{ x \} \times N : x \in \mathbb{T}^2 \} \), we have
\begin{align*}
\Leb_{N}(f^n(\Lambda)) &= \int_{\Lambda} \left| \det \left( Df^n \big|_{TN(x_0,y)} \right) \right| \, d\Leb_{N} \\
&\geq e^{\frac{an}{2}} \, \Leb_{N}(\Lambda).
\end{align*}
Thus, when \( n \) is sufficiently large, this will contradict the fact that the volume of \( N \) is finite.
\end{proof}

    Since the physical measure obtained in the mixed setting has only positive Lyapunov exponents along \( E^{cu} \) \cite[Theorem~A]{MiCaoYang2017}, it follows that   to construct examples that belong to our case, there should exist a proper subbundle of \( TN \) that is mostly expanding, and there should exist another proper subbundle of \( TN \) that is mostly contracting.  Moreover, since we are considering the change in the number of physical measures,  roughly speaking,  the skew product, compared to derived from Anosov diffeomorphisms, is a good choice to reflect this change.  
All these reasons lead to the conclusion that the characteristics of many classical examples naturally emerge in the examples of Theorem~\ref{main}.


We accomplish this work as follows: In Section~\ref{two}, we introduce the background material needed for this paper. We then construct a \(C^\infty\) diffeomorphism $f$ in Section~\ref{three} and explicitly identify its skeleton in Section~\ref{five}. For any point \(\tilde{d} \in \mathrm{Skeleton}\), the closure of \(W^u(\tilde{d})\) always contains the two points that make inequality~\eqref{liangdian} hold.  Next, within any \(C^1\)-neighborhood of \(f\), we construct a \(C^\infty\) diffeomorphism whose skeleton has cardinality one in Section~\ref{six}. Finally, we show that the constructed \(f\) has a mixed center in Section~\ref{seven} and \ref{eight}. The proof of Theorem~\ref{main} is summarized in the final section~\ref{nine}.

\section{Definitions and Tools Used in This Work}\label{two}

In this section, we assume that \(f\) admits a partially hyperbolic splitting
\[
TM = E^{uu} \oplus_{\succ} E^{cu} \oplus_{\succ} E^{cs}.
\]

\subsection{Some Properties of Cone Fields and Domination}

Let \(E\) and \(F\) be two continuous subbundles with trivial intersection, i.e., their intersection consists only of the zero vector. For simplicity, we will omit the base points of tangent vectors in what follows.
For convenience, we introduce the following cone, which differs slightly from the classical definition (since we require certain subbundles to be invariant).  
For any \( \alpha > 0 \), the \textbf{ cone field} with respect to \( E \) and \( F \) is defined by
\[
\mathscr{C}_{\alpha}(E,F)
=  \{ v_1 + v_2 \in E \oplus F : \|v_2\| \le \alpha \|v_1\|, v_1\in E,v_2\in F \}
\]
 whenever
\begin{equation}\label{cones1}
Df(E \oplus F)=E \oplus F
\quad\text{and}\quad
Df(F)=F.
\end{equation}
Moreover, we say that the cone $\mathscr{C}_{\alpha}(E,F)$ is \textbf{forward-invariant}  (resp. \textbf{backward-invariant}) under $f$  if there exists a constant \(\kappa \in(0,1)\) such that
\[
Df\bigl(\mathscr{C}_{\alpha}(E,F)\bigr)
\subset \mathscr{C}_{\kappa \alpha}(E,F) \quad\text{(resp. }
Df^{-1}\bigl(\mathscr{C}_{\alpha}(E,F)\bigr)
\subset \mathscr{C}_{\kappa \alpha}(E,F)\text{)}.
\]

\begin{bsc}\label{fact0}
Let \( \mathbb{R} \) denote the set of real numbers.  
Let
\[
h : R_1 \times R_2 \longrightarrow R_1 \times R_2
\]
be a \( C^1 \) diffeomorphism, where \( R_1 = R_2 = \mathbb{R} \).
Assume that the tangent map of \( h \) has the form
\[
Dh(x) =
\begin{pmatrix}
\chi_{11}(x) & 0 \\
\chi_{12}(x) & \chi_{22}(x)
\end{pmatrix},
\]
and satisfies
\begin{equation}\label{youjie111}
\sup_{x \in \mathbb{R}^2} \frac{|\chi_{22}(x)|}{|\chi_{11}(x)|} < 1,
\qquad
\sup_{x \in \mathbb{R}^2} \frac{|\chi_{12}(x)|}{|\chi_{11}(x)|} < +\infty .
\end{equation}
For clarity, \( \mathscr{C}_{\alpha}(R_1, R_2) \) is defined by 
\[
\left\{ \begin{pmatrix} x_1 \\ x_2 \end{pmatrix} : |x_2| \leq \alpha |x_1| \right\}.
\]
\begin{lem}\label{fact1}
We have the following:
\begin{itemize}
    \item There exists \( \alpha > 0 \) such that the cone field
    \[
    \mathscr{C}_{\alpha}(R_1, R_2)
    \]
    is forward invariant under \( h \).
    \item When \( \chi_{11}(x) = \lambda \) for all \( x \in \mathbb{R}^2 \) and \( \mathscr{C}_{\alpha}(R_1, R_2) \) is forward invariant, then
    \[
    \frac{\lambda^n}{\sqrt{1+\alpha^2}} \leq \| Dh^n|_{\mathscr{C}_{\alpha}(R_1, R_2)} \| \leq \lambda^n \sqrt{1 + \alpha^2}.
    \]
In particular, 
\[
\frac{\lambda^n}{\sqrt{1+\alpha^2}} \leq \| Dh^n |_{\mathscr{C}_{\alpha}(R_1, R_2)} \|
\]
holds true for all \( n \geq 1 \), even without the assumption that \( \mathscr{C}_{\alpha}(R_1, R_2) \) is forward invariant.
    \item When
\[
0 < \chi_{\min} \leq |\chi_{11}(x)| \leq \chi_{\max} \quad \text{for all } x \in \mathbb{R}^2,
\]
where \( \chi_{\min} \) and \( \chi_{\max} \) are positive constants, and \( \mathscr{C}_{\alpha}(R_1, R_2) \) is forward invariant, then
    \[
    \frac{\chi_{\min}^{n}}{\sqrt{1 + \alpha^{2}}} \le \bigl\| Dh^{n} \big|_{\mathscr{C}_{\alpha}(R_1, R_2)} \bigr\| \le \chi_{\max}^{n} \sqrt{1 + \alpha^{2}}.
    \]
\end{itemize}
\end{lem}
\begin{proof}
For any vector \(v=\begin{pmatrix}1\\ \epsilon_2\end{pmatrix} \in \mathscr{C}_{\alpha}(R_1, R_2)\),\[
Dh(x)v = 
\begin{pmatrix}
\chi_{11}(x) & 0 \\
\chi_{12}(x) & \chi_{22}(x)
\end{pmatrix}
\begin{pmatrix}
1 \\
\epsilon_2
\end{pmatrix} 
= \begin{pmatrix} 
\chi_{11}(x) \\
\chi_{12}(x) + \chi_{22}(x)\epsilon_2 
\end{pmatrix}.
\]
Thus, in order to ensure that \( \mathscr{C}_{\alpha}(R_1, R_2) \) is forward invariant, we need to find \( \alpha > 0 \) and \( \kappa \in (0, 1) \) such that:
\begin{equation}\label{zhengx}
\frac{|\chi_{12}(x) + \chi_{22}(x)\epsilon_2|}{|\chi_{11}(x)|} \leq \alpha \kappa.
\end{equation}
However,
\[
\frac{|\chi_{12}(x) + \chi_{22}(x)\epsilon_2|}{|\alpha \chi_{11}(x)|} 
\leq \sup_{x \in \mathbb{R}^2} \frac{|\chi_{22}(x)|}{|\chi_{11}(x)|} + \sup_{x \in \mathbb{R}^2} \frac{|\chi_{12}(x)|}{\alpha |\chi_{11}(x)|}.
\]
It follows from assumption~\eqref{youjie111} that we can choose \( \alpha \) large enough to guarantee the existence of \( \kappa \in (0, 1) \) such that inequality~\eqref{zhengx} holds.

Next, we prove the second item. For any vector \( v = \begin{pmatrix} 1 \\ \epsilon_2 \end{pmatrix} \in \mathscr{C}_{\alpha}(R_1, R_2) \), combine the Pythagorean theorem and the forward invariance, the value of \( \| Dh^n(v) \| \) satisfies
\[
\lambda^n \leq \| Dh^n(v) \| \leq \lambda^n \sqrt{1+\alpha^2}, \quad \text{and} \quad 1 \leq \| v \| \leq \sqrt{1+\alpha^2}.
\]
Consequently, the operator norm of \( Dh^n \) restricted to the cone \( \mathscr{C}_{\alpha}(R_1, R_2) \) satisfies
\[
\frac{\lambda^n}{\sqrt{1+\alpha^2}} \leq \| Dh^n|_{\mathscr{C}_{\alpha}(R_1, R_2)} \| \leq \lambda^n \sqrt{1+\alpha^2}.
\]

The third item is similar to the estimation in the second item. We obtain
\[
\frac{\chi_{\min}^{n}}{\sqrt{1+\alpha^{2}}}
\le
\bigl\| Dh^{n} \big|_{\mathscr{C}_{\alpha}(R_1, R_2)} \bigr\|
\le
\chi_{\max}^{n} \sqrt{1+\alpha^{2}}.
\]
\end{proof}
\end{bsc}
\begin{bsc}\label{fact2}
Let
\[
A(x)=
\renewcommand{\arraystretch}{2.9}
\begin{pmatrix}
\bar{\lambda} & 0 & 0 & 0 \\
\gamma_{21}(x) & \gamma_{22}(x) & \gamma_{23}(x) & \gamma_{24}(x) \\
\gamma_{31}(x) & \gamma_{32}(x) & \gamma_{33}(x) & \gamma_{34}(x)\\
\gamma_{41}(x) & \gamma_{42}(x) & \gamma_{43}(x) & \gamma_{44}(x) 
\end{pmatrix},
\qquad
v=\begin{pmatrix}1\\ \epsilon_2\\ \epsilon_3\\ \epsilon_4\end{pmatrix}\in\mathbb{R}^4.
\]
Suppose that 
\begin{equation}\label{beikongzhi}
|\gamma_{ij}(x)| \le M < +\infty \quad \text{for all } i,j,x, \quad \text{and} \quad \sqrt{\sum_{k=2,3,4} |\epsilon_k|^2} \le \alpha.
\end{equation}
It follows from the calculations that:
\[
A(x)v=\renewcommand{\arraystretch}{2.9}
\begin{pmatrix}
\bar{\lambda}\\
\gamma_{21}(x)+\gamma_{22}(x)\epsilon_2+\cdots+\gamma_{24}(x)\epsilon_4\\
\gamma_{31}(x)+\gamma_{32}(x)\epsilon_2+\cdots+\gamma_{34}(x)\epsilon_4\\
\gamma_{41}(x)+\gamma_{42}(x)\epsilon_2+\cdots+\gamma_{44}(x)\epsilon_4
\end{pmatrix}.
\]
By our setup in~\eqref{beikongzhi}, there exists \( \widetilde{M} > 0 \) such that
\[
\sqrt{\sum_{i=1,2,3} \left( \left| \gamma_{i1}(x) \right| + \sum_{j=2}^4 \left| \gamma_{ij}(x) \right| \cdot |\epsilon_j| \right)^2} \leq \widetilde{M}.
\]
The following diagram shows the different angles corresponding to various values of \( \bar{\lambda} \).
\begin{equation}
	\includegraphics[width=1\textwidth]{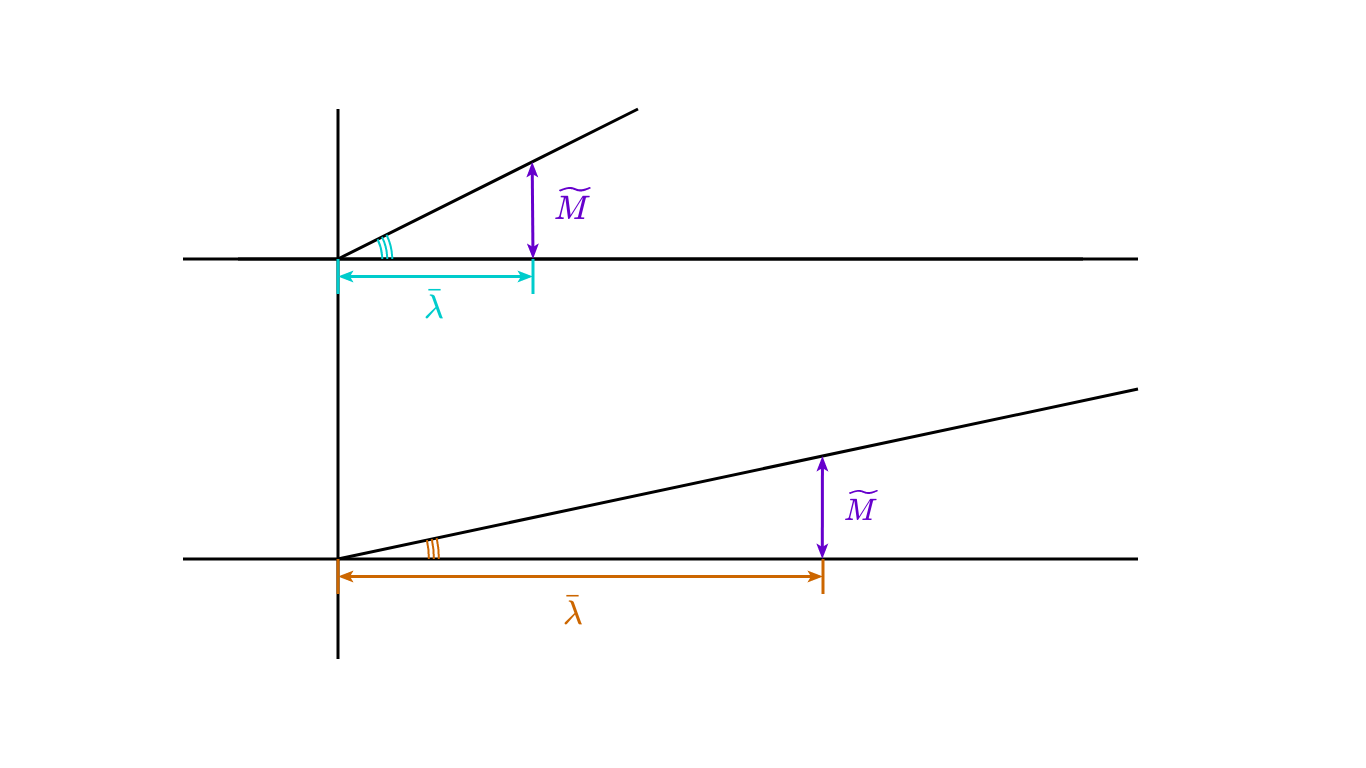}
\end{equation}
Thus, we can always select \( \bar{\lambda} \) sufficiently large such that \( \mathscr{C}_{\alpha}(\mathbb{R}, \mathbb{R}^3) \) is forward invariant under every matrix \( A(x) \).  Analogously to Lemma~\ref{fact1}, after fixing \( \bar{\lambda} \) such that the cone is forward invariant, for every
\(v \in \mathscr{C}_{\alpha}(\mathbb{R},\mathbb{R}^3)\) we have
\[
\|A^n(x) v\|\ \ge\ \frac{\bar{\lambda}^{n}}{\sqrt{1+3\alpha^2}} \|v\|.
\]
\end{bsc}

We also have
\begin{lem}\cite{Yoccoz}\label{cone}
If \(\mathscr{C}_{\alpha}(E,F)\) is forward-invariant (resp. backward-invariant) under $f$, then there exists a  dominated splitting
\[
E\oplus F=\widetilde{E} \oplus_\succ F
\quad (\text{resp. } E\oplus F=F \oplus_\succ \widetilde{E})
\]
such that
\[
\widetilde{E} \subset \mathscr{C}_{\alpha}(E,F)
\quad
(\text{resp. } \widetilde{E} \subset \mathscr{C}_{\alpha}(E,F)).
\]
\end{lem}

\begin{lem}\cite[Exercise in Chapter 4]{Wen2016}\label{zhibiao}
Let \( f \) admit two dominated splittings
\[
TM = E_1 \oplus_{\succ} F_1 = E_2 \oplus_{\succ} F_2.
\]
\begin{enumerate}
    \item If \( \dim E_1 < \dim E_2 \), then \( E_1 \subset E_2 \).
    \item If \( \dim E_1 = \dim E_2 \), then \( E_1 = E_2 \) and \( F_1 = F_2 \).
\end{enumerate}
\end{lem}

\begin{lem}\label{mutually}
If \( f \) has dominated splittings \( G_1 \oplus_{\succ} G_2 \) and \( G_2 \oplus_{\succ} G_3 \), then we have \( G_1 \oplus_{\succ} (G_2 \oplus G_3) \).  By considering the reverse direction (which is similar to the forward direction), we naturally also have \( (G_1 \oplus G_2) \oplus_\succ G_3 \).
\end{lem}
\begin{proof}
It is clear that we only need to prove that there exist some \( \widehat{c} > 0 \) and \( \widehat{\sigma} < 1 \) such that, for any vector \( v_i \in G_i \), \( i = 1, 2, 3 \), and for all \( n \geq 1 \), we have
\begin{equation}\label{mubiaodomin}
\frac{\| Df^n (v_2 + v_3) \|}{\| Df^n (v_1) \|} \leq \widehat{c} \, \widehat{\sigma}^n \frac{\| v_2 + v_3 \|}{\| v_1 \|}.
\end{equation}

Without loss of generality, we assume that the angles \( \langle G_2, G_3 \rangle \) between \( G_2 \) and \( G_3 \) are uniformly bounded below by \( \theta > 0 \), where \( \theta \in (0, \frac{\pi}{2}] \).
 It is clear that there exist \( c > 0 \) and \( \sigma \in (0, 1) \) such that 
\begin{equation}\label{jiajiaodomin}
\frac{\| Df^n (v_2) \|}{\| Df^n (v_1) \|} \leq c \sigma^n \frac{\| v_2 \|}{\| v_1 \|}, \quad \text{and} \quad \frac{\| Df^n (v_3) \|}{\| Df^n (v_1) \|} \leq c \sigma^n \frac{\| v_3 \|}{\| v_1 \|}.
\end{equation}
\begin{equation}
	\includegraphics[width=0.8\textwidth]{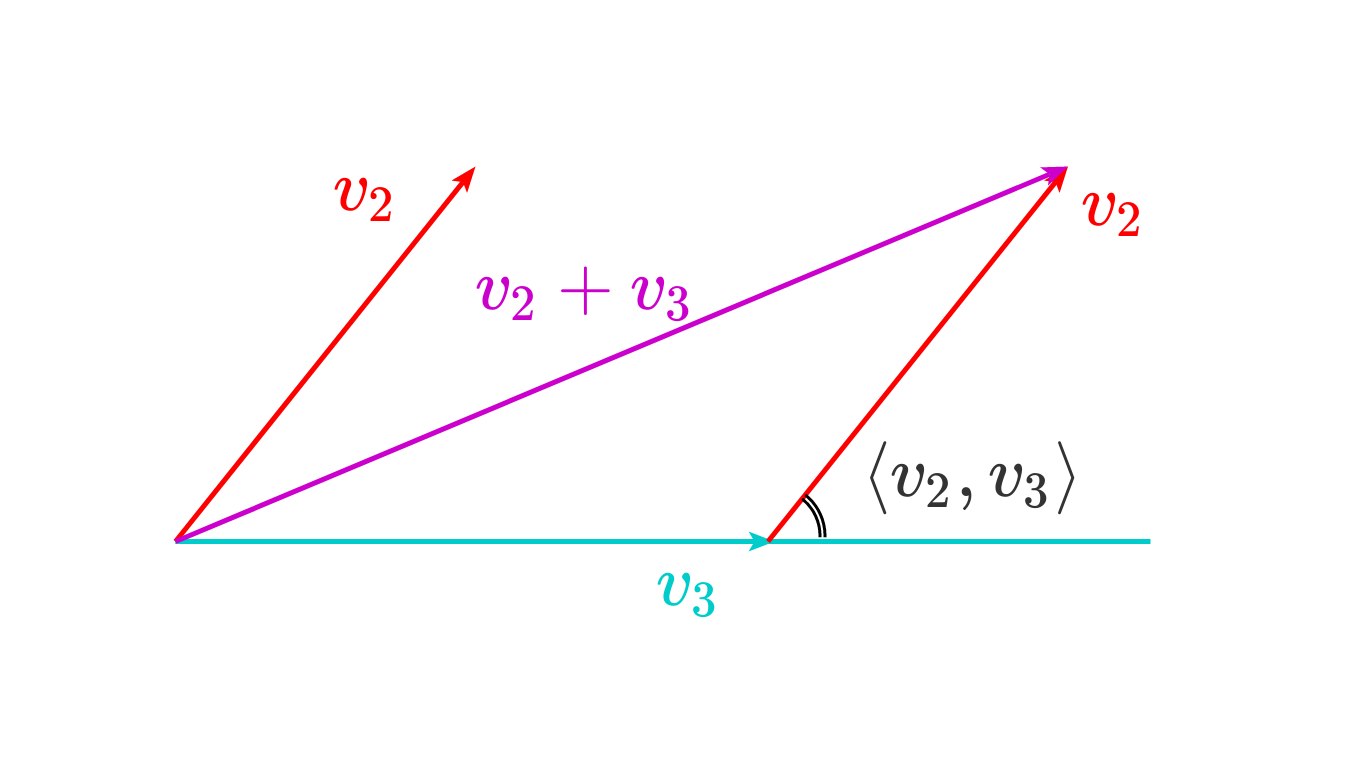}
\end{equation}
By the law of cosines, we have
\[
\| v_2 + v_3 \|^2 = \| v_2 \|^2 + \| v_3 \|^2 - 2 \| v_2 \| \| v_3 \| \cos(\pi - \langle v_2, v_3 \rangle).
\]
It follows that
\begin{align*}
\| v_2 + v_3 \|^2 &\geq \| v_2 \|^2 + \| v_3 \|^2 - 2 \| v_2 \| \| v_3 \| \cos \theta \\
&\geq (1 - \cos \theta) (\| v_2 \|^2 + \| v_3 \|^2) \\
&\geq \frac{(1 - \cos \theta)}{2} \left( \| v_2 \| + \| v_3 \| \right)^2.
\end{align*}
Thus, we have
\[
\| v_2 \| + \| v_3 \| \leq \sqrt{\frac{2}{1 - \cos \theta}} \| v_2 + v_3 \|.
\]
By combining this with inequality~\eqref{jiajiaodomin}, we can find the desired constants \( \widetilde{c} \) and \( \widetilde{\sigma} \) in inequality~\eqref{mubiaodomin}.
\end{proof}

\begin{lem}\label{daxiao}
Assume that \(E \oplus_{\succ} F\) is a dominated splitting. Then for any unit vectors \(v_E \in E\) and \(v_F \in F\), one has
\[
\frac{1}{n}\log \|Df^{n}(v_F)\|
\le
\frac{1}{n}\log \|Df^{n} (v_E)\|
+ \frac{\log c}{n}
+\log \sigma,
\]
where \(c\) and \(\sigma\) are the constants associated with the dominated splitting.
Consequently, for any ergodic measure \( \mu \), the largest Lyapunov exponent of \( \mu \) along \( F \) is strictly smaller than its smallest Lyapunov exponent along \( E \).
\end{lem}

\subsection{Some Properties of Skeletons and the Mixed Center}

Recall that  \textbf{ stable manifold} of $x$ with respect to $f$ is defined by 
 $$ 
 W^s(x, f) := \left\{ y \in M \ \middle| \ \lim_{n \to +\infty} d\big(f^{n}(y), f^{n}(x)\big) = 0 \right\}. 
$$ 
 Similarly,  \textbf{ unstable manifold} of $x$ with respect to $f$ is defined by 
 $$ 
 W^u(x, f) := W^s\big(x, f^{-1}\big). 
 $$ 
The stable manifold of the orbit \( \text{Orb}(p, f) \) of a periodic point \( p \) with respect to \( f \) is given by:
\[
W^s(\text{Orb}(p, f),f): = \{ x \in M \mid \lim_{n \to +\infty}  d\big(f^{n}(x), \text{Orb}(p, f)\big)=0 \}.
\]
Similarly, \[ 
W^u(\text{Orb}(p, f), f) := W^s(\text{Orb}(p, f), f^{-1}).
\] 

A finite set \( S = \{r_1, r_2, \dots, r_k\} \) is called a \textbf{skeleton} with respect to $f$ if each \( r_i \in S \) is a hyperbolic periodic point with stable index \( \dim E^{cs} \), and \( S \) satisfies the following conditions:
\begin{enumerate}[a.]
\item For every \( C^1 \)-disk \( D \) transverse to \( E^{cs} \), there exists a point \( p \in S \) such that \( D \) intersects the stable manifold of the orbit of \( p \) transversely.
\item\label{feihengjie} For distinct \( p, q \in S \), the stable manifold \( W^s(\text{Orb}(p, f)) \) does not intersect the unstable manifold \( W^u(\text{Orb}(q, f)) \).
\end{enumerate}
For a diffeomorphism \( g \) sufficiently close to \( f \), we denote by \( q(g) \) the continuation of a hyperbolic periodic point \( q \) of \( f \). Similarly, the continuation of a skeleton \( S = \{r_1, \dots, r_k\} \) of \( f \) under \( g \) is denoted by
 \[ 
 S(g) = \{r_1(g), \dots, r_k(g)\}. 
 \] 

\begin{pro}\label{sat}\cite[Section 11.2]{CLM}
If \(f\) is of class \(C^{1+}\), then the set of Gibbs \(u\)-states is nonempty, weak\(^*\)-compact, and convex.
Moreover, almost every ergodic component of a Gibbs \(u\)-state is itself a Gibbs \(u\)-state.
\end{pro}

\begin{lem}\cite[Theorem~A,Theorem~B]{MiCao2021}\label{change}
Assume that \(f\) is \(C^{1+}\), and that \(E^{cu}\) is mostly expanding while \(E^{cs}\) is mostly contracting. Then there exists a \(C^1\) neighborhood of \(f\) such that, for every \(C^{1+}\) diffeomorphism \(g\) in this neighborhood, 
\begin{itemize}
\item The map \( g \) admits a skeleton, and the cardinality of the skeleton of \( g \) is less than or equal to the cardinality of \( S(g) \), where \( S(g) \) denotes the continuation of the skeleton of \( f \) under \( g \).  Moreover, when the finite set \( S(g) \) satisfies property~\eqref{feihengjie}, \( S(g) \) is a skeleton of \( g \).
 \item  The number of physical measures of \( g \) equals the cardinality of any skeleton with respect to \( g \). 
\item For any skeleton \( \{ r_1, \dots, r_l \} \) with respect to \( g \), the set of physical measures \( \{ \mu_1, \dots, \mu_l \} \) of \( g \) satisfies the relation
\[
\supp(\mu_i) = \overline{W^u(\text{Orb}(r_i, g), g)} \quad \text{for each } i \in \{ 1, 2, \dots, l \},
\]
where \( \overline{W^u(\text{Orb}(r_i, g), g)} \) denotes the closure of \( W^u(\text{Orb}(r_i, g), g) \).
\end{itemize}
\end{lem}

\section{Construction of Mixed–Kan–Type Skew Products}\label{three}

Fix \(\delta = \frac{1}{10000}\).  Choose a $C^\infty$-smooth function $\psi \colon \mathbb{R} \to \mathbb{R}$ such that:
\begin{itemize}
\item $\psi(x) = \psi(-x)$ for all $x \in \mathbb{R}$ (i.e., $\psi$ is symmetric about $x = 0$);
\item $\psi(x) = 1$ for $x \in \left[ 0, \frac{\delta}{2} \right]$, and $\psi(x) = 0$ for $x \in [\delta, +\infty)$;
\item $\psi(x)$ is strictly monotone on $\left( \frac{\delta}{2}, \delta \right)$.
\end{itemize}

Choose a  $C^\infty$-smooth diffeomorphism \(\phi : \mathbb{R} \to \mathbb{R}\) such that 
\begin{itemize}
\item \(\phi\) is \(2\)-periodic, i.e. \(\phi(x) = \phi(x+2)\) for all \(x \in \mathbb{R}\);
\item on the interval \([0,2]\), \(\phi\) has exactly three fixed points, located at \(0, 1\), and \(2\), satisfying
\begin{equation}\label{contra}
    \phi'(0)=\frac{1}{2} \quad \text{and} \quad \phi'(1) = \frac{3}{2},
   \end{equation}
   and moreover,
\begin{equation}\label{well-defined}
   1 + \min_{x, y \in \mathbb{R}} \bigl\{ (\phi'(x) - 1) \psi(y) \bigr\} > 0.
   \end{equation}
\end{itemize}

Let \( \mathbb{R}^{uu} = \mathbb{R} \) and \( \mathbb{R}^{ss} = \mathbb{R} \). Define
\[
I_\delta : \mathbb{R}^{uu} \times \mathbb{R}^{ss} \times \mathbb{R} \to \mathbb{R}^{uu} \times \mathbb{R}^{ss} \times \mathbb{R}
\]
 by 
\[
I_{\delta}(a,b,c)
:= \Bigl(a,, b,, Q(a,b,c)\Bigr), \quad
Q(a,b,c) = \psi\bigl(\sqrt{a^2+b^2}\bigr)\phi(c) + \bigl(1-\psi(\sqrt{a^2+b^2})\bigr) c.
\]
Let \(B_\gamma(0,0)\) denote the open ball in \(\mathbb{R}^2\) of radius \(\gamma\) centered at \((0,0)\).
We then obtain the following:
\begin{itemize}
\item It follows from condition~\eqref{well-defined} and the computation
\[
\frac{\partial Q}{\partial c}(a,b,c)
=1+(\phi'(c)-1)\psi\bigl(\sqrt{a^{2}+b^{2}}\bigr)>0
\]
that \(I_\delta\) is a \(C^\infty\)-diffeomorphism.
\item  For \((a,b) \in B_{\frac{\delta}{2}}(0,0)\), we have
$I_\delta(a,b,c) = (a,b,\phi(c))$. 
\item For \((a,b) \notin B_\delta(0,0)\), we have \(I_\delta(a,b,c) = (a,b,c)\), i.e., \(I_\delta\) coincides with the identity map.
\item The function \( Q(a, b, c) \), where \( a, b \in \mathbb{R} \), satisfies the following  conditions:
\[
Q(a, b, 0) = 0 \quad \text{and} \quad Q(a, b, 1) = 1.
\]
\item We have
\begin{equation}\label{symmetry}
\iint_{[-\delta,\delta]\times[-\delta,\delta]} (\log\frac{\partial Q}{\partial c}(a,b,0)+\log\frac{\partial Q}{\partial c}(a,b,1) ) da  db<0,
\end{equation}
which holds by condition~\eqref{contra}.
\end{itemize}

Let \(A\) be the toral hyperbolic automorphism induced by
\[
\begin{pmatrix} 2 & 1 \\ 1 & 1 \end{pmatrix}.
\]

Fix a sufficiently small \(\beta > 0\) and a sufficiently large integer \(n_0\) such that the following conditions are satisfied:
\begin{itemize}
\item The map \(A^{n_0}\) admits four fixed points \(p_1, p_2, q_1, q_2\), whose pairwise distances are all greater than \(10\delta\).
\item 
\begin{equation}\label{xiaodeg}
\sigma_0 \ge10000\cdot\max \Biggl\{ \max\Bigl\{ \frac{\partial Q}{\partial c}(a,b,c) : a,b,c \in \mathbb{R} \Bigr\}, \Bigl( \min\Bigl\{ \frac{\partial Q}{\partial c}(a,b,c) : a,b,c \in \mathbb{R} \Bigr\} \Bigr)^{-1} \Biggr\},
\end{equation}
 where \(\sigma_0\) is the largest eigenvalue of \(A^{n_0}\).
\item 
\[
\beta \le \frac{\delta}{2} \cdot \frac{1}{10000}, \quad 2\beta^2 \log \left( \frac{3}{4} \right) + (1 - 2\beta^2) \log \sigma_0 > 0.
\]
\item 
\[
\frac{1}{2} - \frac{1}{10000} < \phi'(x) < \frac{1}{2} + \frac{1}{10000} \quad \text{for all} \quad x \in [-\beta, \beta].
\]
\end{itemize}

Let \( \mathbb{R}^{u} = \mathbb{R} \) and \( \mathbb{R}^{s} = \mathbb{R} \).  Fix a sufficiently large constant \(k\) such that
\[
\frac{3\delta/2}{k} \le \frac{\beta}{10000}.
\]
Define
\[
J_\beta : \mathbb{R}^{uu} \times \mathbb{R}^{ss} \times \mathbb{R} \times \mathbb{R}^{u} \times \mathbb{R}^{s} \to \mathbb{R}^{uu} \times \mathbb{R}^{ss} \times \mathbb{R}\times\mathbb{R}^{u} \times \mathbb{R}^{s} 
\]
 by
\[
J_{\beta}(a,b,c,d,e)
:= (a,b, Q(a,b,c),P(a,b,c,d,e),e)
\]
where 
\[
\begin{aligned}
P(a,b,c,d,e)
&:= \bigl(1 - \psi(k d) \cdot \psi(k \sqrt{a^2 + b^2 + c^2 + e^2})\bigr)\cdot d+ \psi(k d) \cdot \psi(k \sqrt{a^2 + b^2 + c^2 + e^2}) \cdot \frac{3d}{4\sigma_0}\\
  &= d + \psi(k d) \cdot \psi(k \sqrt{a^2 + b^2 + c^2 + e^2}) \cdot \Bigl(\frac{3d}{4\sigma_0} - d\Bigr).
  \end{aligned}
 \]

 A simple calculation gives
\begin{equation}\label{welldefined2}
\sigma_0\frac{\partial P}{\partial d}(a,b,c,d,e)=\sigma_0 + \Bigl(\frac{3}{4} - \sigma_0\Bigr)\psi(k \sqrt{a^2 + b^2 + c^2 + e^2})\bigl(\psi(k d) + \psi'(k d)kd \bigr) \ge \frac{3}{4} .
\end{equation}
By combining inequalities~\eqref{welldefined2} and~\eqref{xiaodeg}, it follows that
\begin{equation}\label{welldefined3}
\min\left\{ \sigma_0^2 \frac{\partial P}{\partial d}(a,b,c,d,e):a,b,c,d,e\in   \mathbb{R} \right\} \ge 7500.
\end{equation}

From this construction, we see that
\begin{itemize}
\item It follows from condition~\eqref{welldefined2}
that \(J_{\beta}\) is a \(C^\infty\)-diffeomorphism.
\item  For  $(a,b,c)\in \Bigl[-\frac{\beta}{2},\frac{\beta}{2}\Bigr]^3$, we have
\begin{equation}\label{zhongjian1}
J_{\beta}(a,b,c,d,e) = (a,b,\phi(c),P(a,b,c,d,e),e) \quad\text{and}\quad \sigma_0\frac{\partial P}{\partial d}(a,b,c,d,e)\ge\frac{3}{4}>
   \frac{1}{2}+\frac{1}{10000}
   >\phi'(c).
\end{equation}
(By calculation, it  is easy to see that the support of \( P(a,b,c,d,e) \) is at most in the region
\[
\{(a,b,c,d,e) : |d| \le \frac{2\beta}{30000}, \sqrt{a^2 + b^2 + c^2 + e^2} \le \frac{2\beta}{30000}\}).
\]
\item For \((a,b,c) \notin  \Bigl[-\frac{\beta}{2},\frac{\beta}{2}\Bigr]^3$, we have 
\begin{equation}\label{zhongjian2}
J_{\beta}(a,b,c,d,e) = (a,b,Q(a,b,c),d,e),
\qquad \text{together with relation~\eqref{xiaodeg}.}
\end{equation}
\end{itemize}
\begin{equation}
	\includegraphics[width=1.1\textwidth]{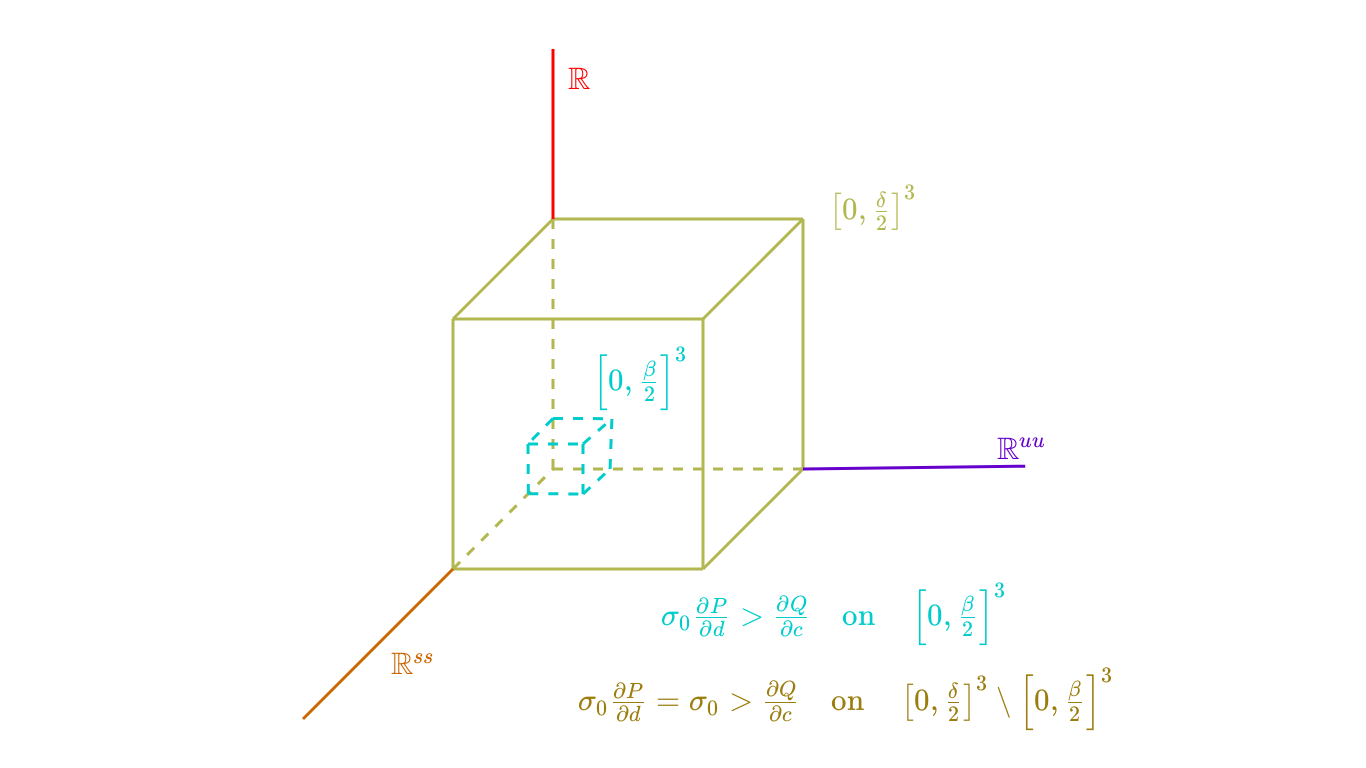}
\end{equation}

Fix a sufficiently large integer \(n_1\) such that the following holds: the largest eigenvalue \(\sigma_1\) of \(A^{n_1}\) satisfies
\begin{equation}\label{big}
\sigma_1 \gg \max\Biggl\{ \sigma_0 \frac{\partial P}{\partial a}, \sigma_0 \frac{\partial P}{\partial b}, \sigma_0 \frac{\partial P}{\partial c}, \sigma_0 \frac{\partial P}{\partial d}, \sigma_0 \frac{\partial P}{\partial e}, \frac{\partial Q}{\partial a}, \frac{\partial Q}{\partial b}, \frac{\partial Q}{\partial c} \Biggr\}.
\end{equation}
(The symbol \( \gg \) indicates that the quantity on the left is "much greater than" the quantity on the right, as it corresponds to \( \bar{\lambda} \) in  Fact~\ref{fact2}.)

Denote by \(\mathcal{F}^{uu}\) and \(\mathcal{F}^{ss}\) the unstable and stable foliations associated with the map \(A^{n_1}\) on \(\mathbb{T}^2\), and by \(\mathcal{F}^{u}\) and \(\mathcal{F}^{s}\) the unstable and stable foliations associated with the map \(A^{n_0}\) on \(\mathbb{T}^2\). Lifting these foliations to \(\mathbb{R}^2\), we denote them by
\(\hat{\mathcal{F}}^{uu}, \hat{\mathcal{F}}^{ss}, \hat{\mathcal{F}}^{u},\) and \(\hat{\mathcal{F}}^{s}\), respectively. 
Moreover, these foliations are tangent everywhere to the subbundles \( E^{uu}, E^{ss}, E^u \), and \( E^s \), respectively. A point \(x\) on the torus $\T^2$, when lifted to \(\mathbb{R}^2\), is denoted by \(\hat{x}\).

The map \( \pi_2 : \mathbb{R} \to [-1, 1) \) is defined by \( \pi_2(c) = \bar{c} \), where \( \bar{c} \in [-1, 1) \) is the unique value such that there exists an integer \( l \in \mathbb{Z} \) satisfying \( \bar{c} - c = 2l \).

Next, we identify
\[
\hat{\mathcal{F}}^*(\hat{p}_i), \ \hat{\mathcal{F}}^*(\hat{q}_i) \quad \text{with} \quad \mathbb{R}^*, \quad \text{for } i=1,2 \text{ and } * \in \{uu, ss, u, s\}.
\]
Let \( \mathbb{S} = \mathbb{R}/2\mathbb{Z} \).
Now, we define a map \( \widetilde{ID} \) on \( \mathbb{T}^2 \times \mathbb{S} \times \mathbb{T}^2 \), such that its lift, \( \widehat{ID} \), satisfies the following properties:
\begin{itemize}
\item When 
\[
(a,b,c,d,e)\in \Lambda(\hat{p}_1)
= \hat{\mathcal{F}}^{uu}_{5\delta}(\hat{p}_1) \times 
  \hat{\mathcal{F}}^{ss}_{5\delta}(\hat{p}_1) \times 
  \mathbb{R} \times 
  \hat{\mathcal{F}}^{u}_{5\delta}(\hat{p}_1) \times 
  \hat{\mathcal{F}}^{s}_{5\delta}(\hat{p}_1),
\]
we define
\[
\widehat{ID}(a,b,c,d,e) = \bigl(a, b, Q(a,b,c), P(a,b,\pi_2(c),d,e), e \bigr),
\]
where \((\hat{p}_1,0,\hat{p}_1)\) is regarded as the origin \((0,0,0,0,0)\).
\item When 
\[
(a,b,c,d,e) \in \Lambda(\hat{p}_2)
= \hat{\mathcal{F}}^{uu}_{5\delta}(\hat{p}_2) \times 
  \hat{\mathcal{F}}^{ss}_{5\delta}(\hat{p}_2) \times 
  \mathbb{R} \times 
  \hat{\mathcal{F}}^{u}_{5\delta}(\hat{p}_2) \times 
  \hat{\mathcal{F}}^{s}_{5\delta}(\hat{p}_2),
\]
we define
\[
\widehat{ID}(a,b,c,d,e) = \bigl(a, b, Q(a,b,c-1), P(a,b,\pi_2(c-1),d,e), e \bigr),
\]
where \((\hat{p}_2, 1, \hat{p}_2)\) is regarded as the point \((0,0,1,0,0)\).
\item When 
\[
(a,b,c,d,e) \in \Gamma(\hat{q}_1)
= \hat{\mathcal{F}}^{uu}_{5\delta}(\hat{q}_1) \times 
  \hat{\mathcal{F}}^{ss}_{5\delta}(\hat{q}_1) \times 
  \mathbb{R} \times 
  \hat{\mathcal{F}}^{u}_{5\delta}(\hat{q}_1) \times 
  \hat{\mathcal{F}}^{s}_{5\delta}(\hat{q}_1),
\]
we define
\[
\widehat{ID}(a,b,c,d,e) = \bigl(a, b, Q(a,b,c), d, e \bigr),
\]
where \((\hat{q}_1, 0, \hat{q}_1)\) is regarded as the origin \((0,0,0,0,0)\).
\item When 
\[
(a,b,c,d,e) \in 
\Gamma(\hat{q}_2)
= \hat{\mathcal{F}}^{uu}_{5\delta}(\hat{q}_2) \times 
  \hat{\mathcal{F}}^{ss}_{5\delta}(\hat{q}_2) \times 
  \mathbb{R} \times 
  \hat{\mathcal{F}}^{u}_{5\delta}(\hat{q}_2) \times 
  \hat{\mathcal{F}}^{s}_{5\delta}(\hat{q}_2),
\]
we define
\[
\widehat{ID}(a,b,c,d,e) = \bigl(a, b, Q(a,b,c-1), d, e \bigr),
\]
where \((\hat{q}_2, 1, \hat{q}_2)\) is regarded as the point \((0,0,1,0,0)\).
\item Outside the neighborhoods described above, we set \(\widehat{ID} (a,b,c,d,e) = \bigl( a,  b, c,  d,  e \bigr)\).
\end{itemize}
Owing to the local nature of the definition of \(P\) and the periodicity of \(Q\), together with the gluing lemma \cite{Lee},  \( \widetilde{ID} \) is well-defined on the entire \(\mathbb{T}^2\times\mathbb{S} \times \mathbb{T}^2 \).  At this point, we define the required smooth diffeomorphism \( f \) by
\[
f = (A^{n_1} \times \mathrm{id} \times A^{n_0}) \circ \widetilde{ID}.
\]
We denote the lift of $f$ by $\hat{f}$.
From the construction we obtain the following
\begin{itemize}
\item The map \( f \) can be written in the form:
\[
f(x,y,z) = \big(f_{\mathbb{T}^2}(x), f_{\mathbb{T}^2 \times \mathbb{S}}(x,y),f_{\mathbb{T}^2 \times \mathbb{S} \times  \mathbb{T}^2}(x,y,z)\big),
\]
where \((x,y,z) \in \mathbb{T}^2 \times \mathbb{S} \times  \mathbb{T}^2\).
\item  \[
f(\mathbb{T}^2 \times \{0\} \times \mathbb{T}^2)
= \mathbb{T}^2 \times \{0\} \times \mathbb{T}^2,
\qquad
f(\mathbb{T}^2 \times \{1\} \times \mathbb{T}^2)
= \mathbb{T}^2 \times \{1\} \times \mathbb{T}^2.
\]
\end{itemize}

We define the following projection maps, which will be used throughout the paper.
First, define
\[
\pi^3_1 : \mathbb{T}^2 \times \mathbb{S} \times \mathbb{T}^2 \longrightarrow \mathbb{T}^2
\]
by
\[
\pi^3_1(x,y,z) = x,
\quad
\text{for all } (x,y,z) \in \mathbb{T}^2 \times \mathbb{S} \times \mathbb{T}^2 .
\]
Next, define
\[
\pi^3_2 : \mathbb{T}^2 \times \mathbb{S} \times \mathbb{T}^2
\longrightarrow
\mathbb{T}^2 \times \mathbb{S}
\]
by
\[
\pi^3_2(x,y,z) = (x,y),
\quad
\text{for all } (x,y,z) \in \mathbb{T}^2 \times \mathbb{S}.
\]
Finally, define
\[
\pi^2_1 : \mathbb{T}^2 \times \mathbb{S} \longrightarrow \mathbb{T}^2
\]
by
\[
\pi^2_1(x,y) = x,
\quad
\text{for all } (x,y) \in \mathbb{T}^2 \times \mathbb{S}.
\]
Define the map
\[
g : \mathbb{T}^2 \times \mathbb{S} \longrightarrow \mathbb{T}^2 \times \mathbb{S}
\]
by
\[
g(x,y) = \big(f_{\mathbb{T}^2}(x), f_{\mathbb{T}^2 \times \mathbb{S}}(x,y)\big), \quad \text{for } (x,y) \in \mathbb{T}^2 \times \mathbb{S}.
\]

Let
\[
\Pi:\widetilde{M}\longrightarrow \widetilde{N}
\]
be a measurable map between measurable spaces \((\widetilde{X},\mathcal{M})\) and \((\widetilde{Y},\mathcal{N})\).
If \(\widetilde{\mu}\) is a measure on \((\widetilde{X},\mathcal{M})\), the \textbf{pushforward measure} \(\Pi_*(\widetilde{\mu})\) on \((\widetilde{Y},\mathcal{N})\) is defined by
\[
\Pi_*(\widetilde{\mu})(B)=\widetilde{\mu}(\Pi^{-1}(B)),
\qquad \text{for every } B\in\mathcal{N}.
\]
In the application of these definitions, \(\Pi\) will correspond to continuous maps, and \(\mathcal{M}\) and \(\mathcal{N}\) will correspond to the Borel \(\sigma\)-algebras.

We emphasize that all the notations defined in this section~\ref{three} will be used in all subsequent chapters. These notations include:  \( \delta \), \( \beta \), \( A^{n_0} \), \( A^{n_1} \), \( \sigma_0 \), \( \sigma_1 \), \( Q \), \( P \), \( \phi \), \( p_i \), \( q_i \), where \( i=1,2 \), \( \mathbb{S} \), \( f \), \( g \), \( E^* \), \( \mathcal{F}^* \), \( \hat{\mathcal{F}}^* \), \(  \mathbb{R}^* \), where \( * \in \{ uu, ss, u, s \} \), \( \pi^3_2 \), \( \pi^2_1 \), and \( \hat{x} \).

\section{Mostly Contracting Behavior of the Subsystem}

We first show that \(g\) is partially hyperbolic.
It is straightforward to see that 
\[
Dg=\renewcommand{\arraystretch}{2.4}
\begin{pmatrix}
\sigma_1 & 0 & 0 \\
\frac{\partial Q}{\partial a}  & \frac{\partial Q}{\partial c}   & \frac{\partial Q}{\partial b} \\
 0 & 0 & \frac{1}{\sigma_1}
\end{pmatrix}.
\]
It follows from our setup~\eqref{xiaodeg} and~\eqref{big}  that there exist positive constants \(\rho_1\) and \(\rho_2\) such that the cones
\[
\mathscr{C}_{\rho_1}(E^{uu}, T\mathbb{S})
\quad\text{and}\quad
\mathscr{C}_{\rho_2}(E^{ss}, T\mathbb{S})
\]
are forward-invariant and backward-invariant under \(g\), respectively. Combining Lemma~\ref{fact1} and Lemma~\ref{cone}, we conclude that the map \(g\) admits a unique partially hyperbolic splitting
\[
T(\mathbb{T}^2 \times \mathbb{S})=\widetilde{E^{uu}_g} \oplus_{\succ} T\mathbb{S} \oplus_{\succ} \widetilde{E^{ss}_g}.
\]


For the sake of exposition, we may assume that the partial lifts of
\(\mathbb{T}^2 \times \mathbb{S}\) and \(\mathbb{T}^2 \times \mathbb{S} \times \mathbb{T}^2\)
are given by \(\mathbb{T}^2 \times \mathbb{R}\) and
\(\mathbb{T}^2 \times \mathbb{R} \times \mathbb{T}^2\), respectively.  The corresponding maps on the partial lifts will still be denoted by
\( g \) and \( f \), respectively.

\begin{lem}\label{mostcon}
We have
\begin{itemize}
\item Any disk \(D^{uu}\) that is transverse to \(T\mathbb{S} \oplus_{\succ} \widetilde{E^{ss}_g}\) must intersect transversely either  \(W^{s}\bigl((q_1,0), g\bigr)\) or
\(W^{s}\bigl((q_2,1), g\bigr)\).
\item The invariant subbundle \(T\mathbb{S}\) is mostly contracting with respect to $g$.  
\item The map \(g\) has exactly two Gibbs \(u\)-states, which are given by the Lebesgue measures supported on \(\mathbb{T}^2 \times \{i\}\), for \(i\in\{0,1\}\). 
\end{itemize}
\end{lem}
\begin{proof}
Since \(g\) acts as a product map on
\(\mathcal{F}^{ss}_{\frac{\delta}{2}}(q_1) \times (-1,1)\), it follows that
\[
\mathcal{F}^{ss}_{\frac{\delta}{2}}(q_1) \times (-1,1) \subset W^s\bigl((q_1,0), g\bigr).
\]
Consequently, the stable manifold of \((q_1,0)\) is given by
\[
W^s\bigl((q_1,0), g\bigr)
= \bigcup_{j \in \mathbb{Z}} g^j \bigl( \mathcal{F}^{ss}(q_1) \times (-1,1) \bigr)
= \mathcal{F}^{ss}(q_1) \times (-1,1).
\]
Similarly, for \((q_2,1)\), we have 
\[
W^s\bigl((q_2,1), g\bigr) = \mathcal{F}^{ss}(q_2) \times (0,2).
\]
Thus, the first item follows immediately.

Notice that for each \(i \in \{0,1\}\), the Lebesgue measure on \(\mathbb{T}^2 \times \{i\}\) is the unique ergodic Gibbs \(u\)-state supported on \(\mathbb{T}^2 \times \{i\}\). Moreover, \(\mathbb{T}^2 \times \{i\}\) is a \(u\)-saturated, \(g\)-invariant compact subset of \(\mathbb{T}^2 \times \mathbb{S}\).  Consequently, by \cite[Proposition~3.3]{DolgopyatVianaYang2016}, it suffices to show that the Lebesgue measure on \(\mathbb{T}^2 \times \{i\}\) has only negative Lyapunov exponents along \(T\mathbb{S}\). 

The following figure illustrates our construction.
\begin{equation}
	\includegraphics[width=1.1\textwidth]{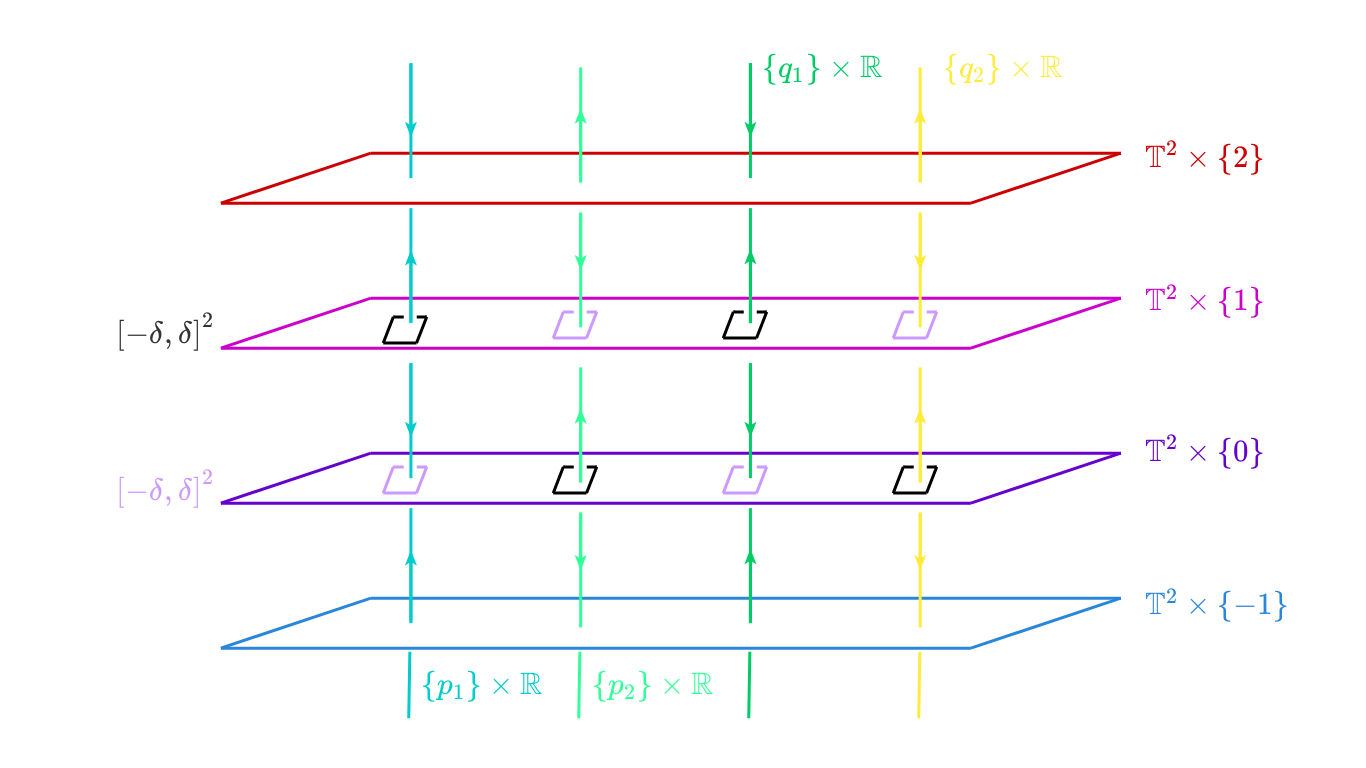}
\end{equation}
Since \((p_1,0)\) and \((p_2,1)\) agree locally with \((q_1,0)\) and \((q_2,1)\), respectively, and \(\det (Dg|_{T\mathbb{S}}) = \frac{\partial Q}{\partial c}\), it is therefore enough to prove that

\begin{equation}\label{symmetry2}
\iint_{[-\delta,\delta]^2_{(\hat{q}_1,0)}} \log \frac{\partial Q}{\partial c}(a,b,0) da db+
\iint_{[-\delta,\delta]^2_{(\hat{q}_2,0)}} \log \frac{\partial Q}{\partial c}(a,b,0) da db < 0, 
\end{equation}
where 
\[
[-\delta,\delta]^2_{(\hat{q}_1,0)} = \hat{\mathcal{F}}^{uu}_\delta(\hat{q}_1) \times \hat{\mathcal{F}}^{ss}_\delta(\hat{q}_1) \quad \text{and} \quad
[-\delta,\delta]^2_{(\hat{q}_2,0)} = \hat{\mathcal{F}}^{uu}_\delta(\hat{q}_2) \times \hat{\mathcal{F}}^{ss}_\delta(\hat{q}_2).
\]
By symmetry, we have
\[
\iint_{[-\delta,\delta]^2_{(\hat{q}_2,0)}} \log \frac{\partial Q}{\partial c}(a,b,0) da db
= \iint_{[-\delta,\delta]^2_{(\hat{q}_1,1)}} \log \frac{\partial Q}{\partial c}(a,b,1) da db,
\]
where
\[
[-\delta,\delta]^2_{(\hat{q}_1,1)} = \hat{\mathcal{F}}^{uu}_\delta(\hat{q}_1) \times \hat{\mathcal{F}}^{ss}_\delta(\hat{q}_1).
\]
It follows from relation~\eqref{symmetry} that condition~\eqref{symmetry2} is satisfied.
\end{proof}


\section{The Partially Hyperbolic Splitting of \(f\) and the Skeleton}\label{five}

We now consider the action of
\[
Df : E^{uu} \oplus E^{u} \oplus T\mathbb{S} \oplus E^{s} \oplus E^{ss}
\longrightarrow
E^{uu} \oplus E^{u} \oplus T\mathbb{S} \oplus E^{s} \oplus E^{ss}.
\]
Thus, when   \(\hat{x} \in \Lambda(\hat{p}_i)\),  \(Df(x)\) is given by  
\begin{equation}\label{exp1}
Df(x)=D\hat{f}(\hat{x})=
\renewcommand{\arraystretch}{2.9}
\begin{pmatrix}
\sigma_1 & 0 & 0 & 0 & 0 \\
\sigma_0\frac{\partial P}{\partial a} & \sigma_0\frac{\partial P}{\partial d}  & \sigma_0\frac{\partial P}{\partial c} & \sigma_0\frac{\partial P}{\partial e} & \sigma_0\frac{\partial P}{\partial b} \\
\frac{\partial Q}{\partial a} & 0 & \frac{\partial Q}{\partial c}  & 0 & \frac{\partial Q}{\partial b} \\
0 & 0 & 0 & \frac{1}{\sigma_0} & 0 \\
0 & 0 & 0 & 0 & \frac{1}{\sigma_1}
\end{pmatrix}
\end{equation}
Similarly,  when   \(\hat{x} \in\Gamma(\hat{q}_i)\),  \(Df(x)\) is given by  
\begin{equation}\label{exp2}
Df(x)=D\hat{f}(\hat{x})=
\renewcommand{\arraystretch}{2.5}
\begin{pmatrix}
\sigma_1 & 0 & 0 & 0 & 0 \\
0 & \sigma_0  & 0 & 0 & 0 \\
\frac{\partial Q}{\partial a} & 0 & \frac{\partial Q}{\partial c}  & 0 & \frac{\partial Q}{\partial b} \\
0 & 0 & 0 & \frac{1}{\sigma_0} & 0 \\
0 & 0 & 0 & 0 & \frac{1}{\sigma_1}
\end{pmatrix}
\end{equation}


By combining relation~\eqref{big} and  Fact~\ref{fact2}, we obtain positive constants \( \tau_1 > 0 \) such that the cones
\begin{equation}\label{tau1}
\mathscr{C}_{\tau_1}(E^{uu}, E^u \oplus T\mathbb{S} \oplus E^s) \quad \text{and} \quad
\mathscr{C}_{\tau_1}(E^{ss}, E^u \oplus T\mathbb{S} \oplus E^s)
\end{equation}
are forward-invariant and backward-invariant under \(f\), respectively.
It follows from Lemma~\ref{cone} that \(f\) admits a partially hyperbolic splitting
\[
TM = \widetilde{E^{uu}} \oplus_\succ \bigl(E^u \oplus T\mathbb{S} \oplus E^s\bigr) \oplus_\succ \widetilde{E^{ss}} ,
\]
where $M=\mathbb{T}^2 \times \mathbb{S} \times \mathbb{T}^2$.   It is easy to see that \( E^u \oplus T\mathbb{S} \oplus E^s \oplus E^{ss}=\bigl(E^u \oplus T\mathbb{S} \oplus E^s\bigr) \oplus_\succ \widetilde{E^{ss}} \).
Then, by Lemma~\ref{mutually},  we obtain 
\begin{equation}\label{dazicong}
 TM = \widetilde{E^{uu}} \oplus_\succ (E^u \oplus T\mathbb{S} \oplus E^s \oplus E^{ss}).
\end{equation}

Since the restriction of \( Df^{-1} \) to \( T\mathbb{S} \oplus E^{u} \) has the form
\[
Df^{-1}\big|_{T\mathbb{S} \oplus E^{u}}
= \begin{pmatrix}
\left( \dfrac{\partial Q}{\partial c} \right)^{-1} & 0 \\
* & \left( \sigma_0 \dfrac{\partial P}{\partial d} \right)^{-1}
\end{pmatrix},
\]
it follows, by combining relations~\eqref{zhongjian1} and~\eqref{zhongjian2} with the first item of Lemma~\ref{fact1}, that there exists a constant \( \tau_3 > 0 \) such that
\[
\mathscr{C}_{\tau_3}(T\mathbb{S}, E^u)
\]
is backward-invariant. It then follows from Lemma~\ref{cone} that the bundle \(T\mathbb{S} \oplus E^u\) admits a dominated splitting
\begin{equation}\label{lianggeES}
T\mathbb{S} \oplus E^u = E^u \oplus_\succ \widetilde{T\mathbb{S}}.
\end{equation}
By the third item of Lemma~\ref{fact1}, we have
\begin{equation}\label{xianghu1}
\frac{1}{\sqrt{1+\tau_3^2}}\max\Bigl\{ \frac{\partial Q}{\partial c}(a,b,c) : a,b,c \in \mathbb{R} \Bigr\}^{-n}\le \bigl\| Df^{-n}\big|_{\widetilde{T\mathbb{S}}} \bigr\|
\le \min\Bigl\{ \frac{\partial Q}{\partial c}(a,b,c) : a,b,c \in \mathbb{R} \Bigr\}^{-n}\sqrt{1+\tau_3^2}.
\end{equation}

The restriction of \( Df^{-1} \) to \( E^s \oplus E^{u} \) has the form
\[
Df^{-1}\big|_{E^s \oplus E^{u}} 
= \begin{pmatrix}
\sigma_0 & 0 \\ 
* & \left( \sigma_0 \dfrac{\partial P}{\partial d} \right)^{-1}
\end{pmatrix}.
\]
By combining the first item of Lemma~\ref{fact1} with inequality~\eqref{welldefined3}, there exists a constant \( \tau_4 > 0 \) such that the cone
$$\mathscr{C}_{\tau_4}(E^s, E^u)$$ is backward-invariant, which, further, by Lemma~\ref{cone},  implies that the bundle \(E^s \oplus E^u\) admits a dominated splitting
\[
E^u \oplus E^s = E^u \oplus_\succ \widetilde{E^s}.
\]
By the second item of Lemma~\ref{fact1}, we have
\begin{equation}\label{xianghu2}
\frac{1}{\sqrt{1+\tau_4^2}} \sigma_0^{n} \le \bigl\| Df^{-n}\big|_{\widetilde{E^s}} \bigr\|
\le \sigma_0^{n} \sqrt{1+\tau_4^2}.
\end{equation}

Because of our setup in inequality~\eqref{xiaodeg}, and by combining the inequalities~\eqref{xianghu1} and~\eqref{xianghu2} that we have just obtained, it follows that:
\[
\widetilde{T\mathbb{S}}\oplus_\succ\widetilde{E^s}.
\]
Thus, based on Lemma~\ref{mutually}, we obtain:
\[
TM = \widetilde{E^{uu}} \oplus_\succ\Bigl( E^u \oplus_\succ \bigl( \widetilde{T\mathbb{S}} \oplus_\succ \widetilde{E^s} \bigr) \Bigr) \oplus_\succ \widetilde{E^{ss}}.
\]
and
\begin{equation}\label{fenjiezicong}
TM = \widetilde{E^{uu}} \oplus_\succ E^u \oplus_\succ \Bigl(\widetilde{T\mathbb{S}} \oplus_\succ \bigl( \widetilde{E^s}  \oplus_\succ \widetilde{E^{ss}}\bigr)\Bigr).
\end{equation}


By computation, we have the following fixed points.

\textbf{The two fixed points with an unstable index of 1:}
\begin{equation}\label{exp3}
Df((p_1,0,p_1))=Df((p_2,1,p_2))=
\renewcommand{\arraystretch}{2}
\begin{pmatrix}
\sigma_1 & 0 & 0 & 0 & 0 \\
0 & \frac{3}{4}  & 0 & 0 & 0 \\
0 & 0 & \frac{1}{2}  & 0 & 0 \\
0 & 0 & 0 & \frac{1}{\sigma_0} & 0 \\
0 & 0 & 0 & 0 & \frac{1}{\sigma_1}
\end{pmatrix}
\end{equation}
\begin{equation}
	\includegraphics[width=1\textwidth]{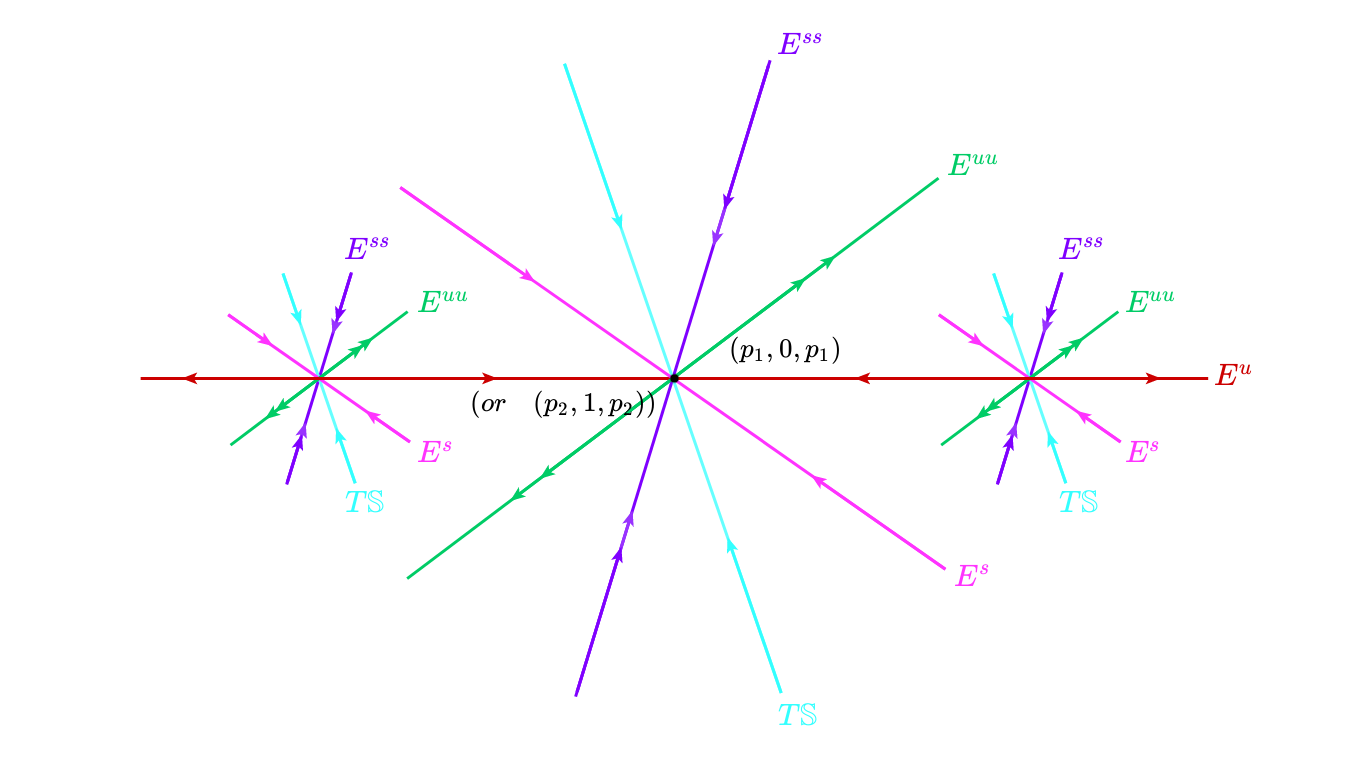}
\end{equation}
\textbf{The two fixed points with an unstable index of 2:}
\begin{equation}\label{exp4}
Df((q_1,0,q_1))=Df((q_2,1,q_2))=
\renewcommand{\arraystretch}{2}
\begin{pmatrix}
\sigma_1 & 0 & 0 & 0 & 0 \\
0 & \sigma_0  & 0 & 0 & 0 \\
0 & 0 & \frac{1}{2}  & 0 & 0 \\
0 & 0 & 0 & \frac{1}{\sigma_0} & 0 \\
0 & 0 & 0 & 0 & \frac{1}{\sigma_1}
\end{pmatrix}
\end{equation}
\begin{equation}
	\includegraphics[width=0.8\textwidth]{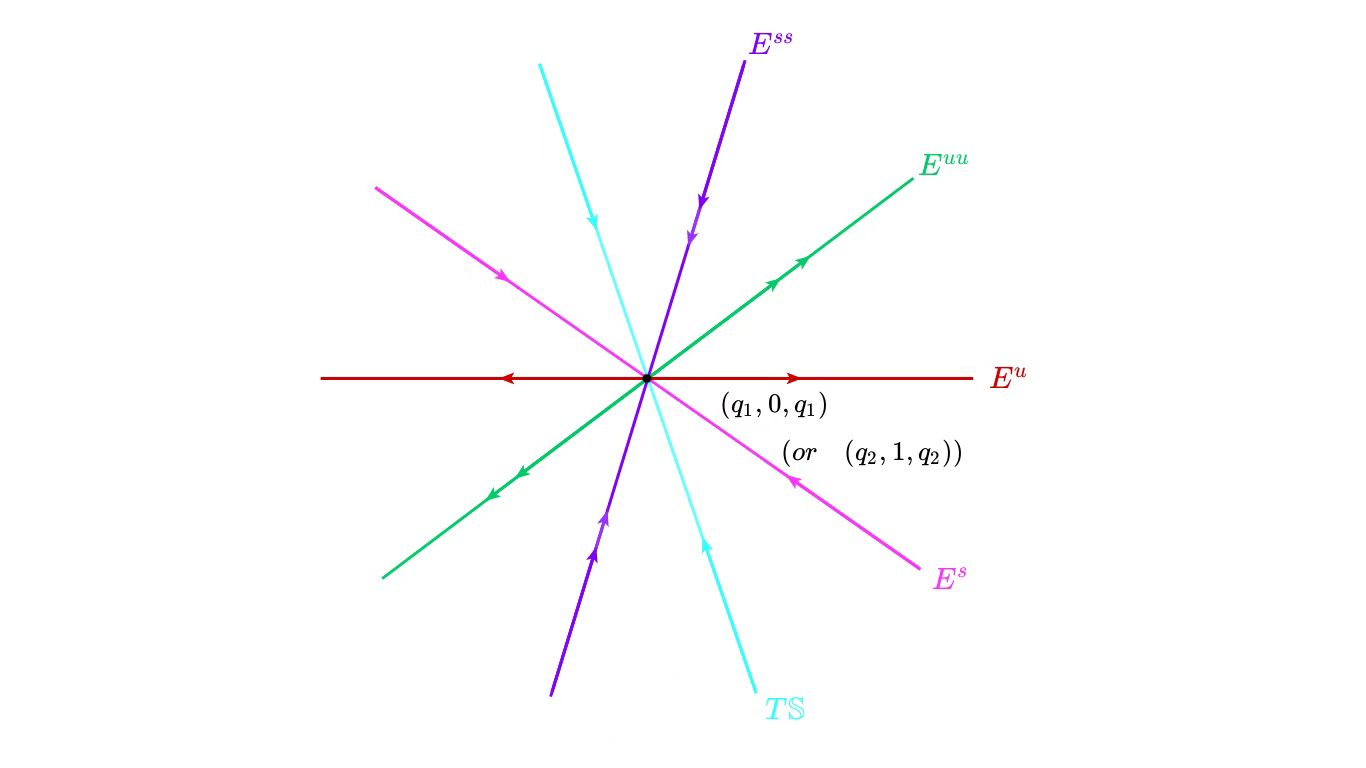}
\end{equation}
\textbf{The two fixed points with an unstable index of 3:}
\begin{equation}\label{exp5}
Df((q_1,1,q_1))=Df((q_2,0,q_2))=
\renewcommand{\arraystretch}{2}
\begin{pmatrix}
\sigma_1 & 0 & 0 & 0 & 0 \\
0 & \sigma_0  & 0 & 0 & 0 \\
0 & 0 & \frac{3}{2}  & 0 & 0 \\
0 & 0 & 0 & \frac{1}{\sigma_0} & 0 \\
0 & 0 & 0 & 0 & \frac{1}{\sigma_1}
\end{pmatrix}
\end{equation}
\begin{equation}
	\includegraphics[width=0.8\textwidth]{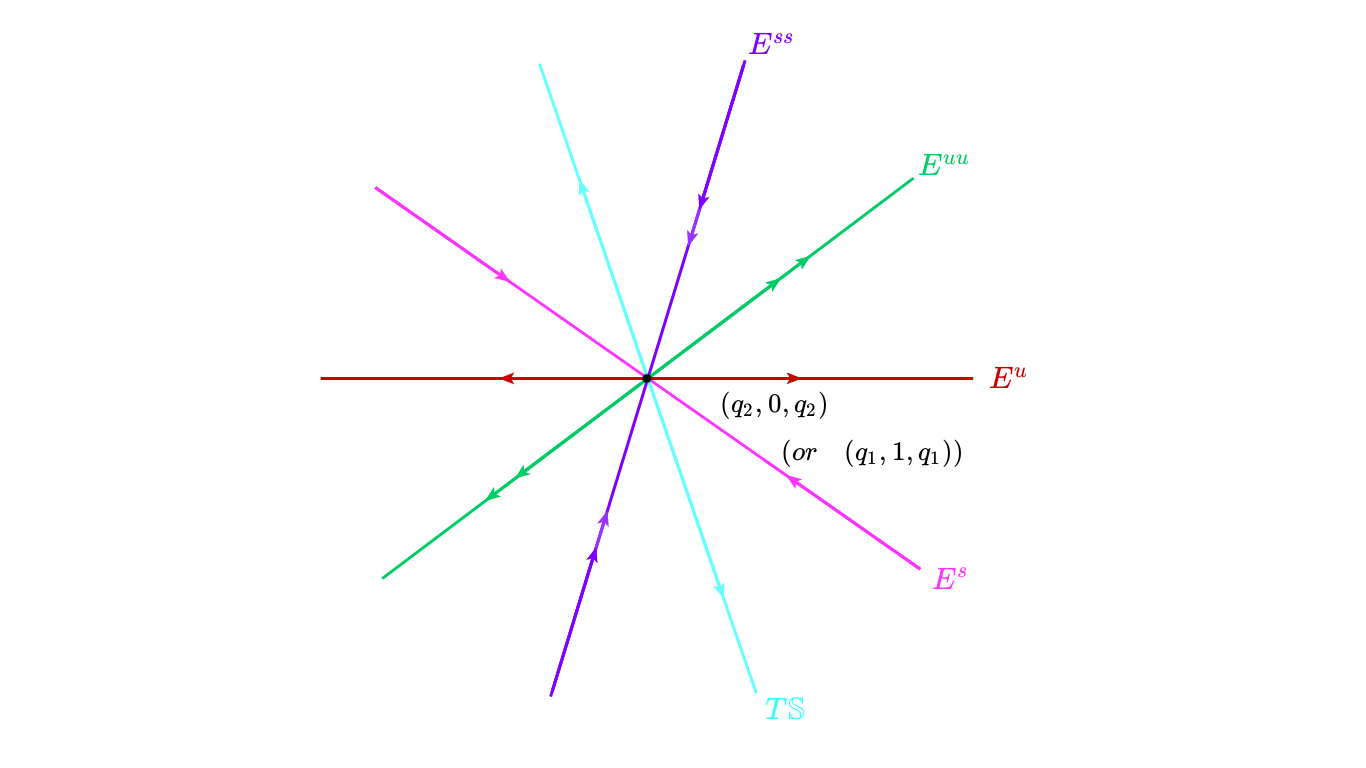}
\end{equation}
At this point, the map \(f\) is not uniformly expanding along \(E^u\). We define 
\[
E^{cu} := E^u, \quad
E^{cs} :=\widetilde{T\mathbb{S}} \oplus_\succ \bigl( \widetilde{E^s}  \oplus_\succ \widetilde{E^{ss}}\bigr).
\]

We now state the following proposition.


\begin{pro}\label{claim}
 For any open $C^1$-disk \(D\) transverse to \(E^{cs}\), if
\[
D \cap \bigl( \mathcal{F}^{ss}(q_1) \times (-1,1) \times \mathbb{T}^2 \bigr) \neq \varnothing \quad (\text{resp. } D \cap \bigl( \mathcal{F}^{ss}(q_2) \times (0,2) \times \mathbb{T}^2 \bigr) \neq \varnothing),
\]
then \(D\) intersects
\[
W^{s}\bigl((q_1,0,q_1), f\bigr) \quad (\text{resp. } W^{s}\bigl((q_2,1,q_2), f\bigr))
\]
transversely.
\end{pro}
Let \(B_\gamma(x,\mathbb{T}^2)\) denote the open ball of radius \(\gamma\) in \(\mathbb{T}^2\) centered at the point \(x\in \mathbb{T}^2\). Since   
\[
f = A^{n_1} \times \phi \times A^{n_0} \quad\text{on}\quad B_{\delta/2}(q_1,\mathbb{T}^2)\times (-1,1) \times \mathbb{T}^2
\]
it follows that 
\[
\mathcal{F}^{ss}_{\delta/2}(q_1) \times (-1,1) \times \mathcal{F}^{s}(q_1)
\subset
W^{s}\bigl((q_1,0,q_1), f\bigr).
\]
 We shall prove only the former case, as the case in parentheses follows by symmetry.
\begin{proof}
Given the product form of \(f\) on the region
$
B_{\delta/2}(q_1,\mathbb{T}^2)\times (-1,1) \times \mathbb{T}^2
$,
it is clear that there exists \(\varepsilon \ll \frac{\delta}{2}\) such that, for any
\[
x_0 \in \mathcal{F}^{ss}_{\delta/100}(q_1), \quad y_0 \in (-1,1), \quad z_0 \in \mathbb{T}^2,
\]
the disk
\[
\mathcal{F}^{uu}_{\varepsilon}(x_0) \times \{y_0\} \times \mathcal{F}^{u}_{\varepsilon}(z_0)
\]
intersects transversely with
\[
\mathcal{F}^{ss}_{\delta/2}(q_1) \times (-1,1) \times \mathcal{F}^s(q_1)
\]
at some point \((x_0, y_0, z_0^s)\), where \(z_0^s \in \mathcal{F}^s(q_1)\cap\mathcal{F}^{u}_{\varepsilon}(z_0)\). This transversality holds because the tangent spaces (on the lifted space) at \((x_0, y_0, z_0^s)\) of
\[
\mathcal{F}^{uu}_{\varepsilon}(x_0) \times \{y_0\} \times \mathcal{F}^{u}_{\varepsilon}(z_0) \quad \text{and} \quad \mathcal{F}^{ss}_{\delta/2}(q_1) \times (-1,1) \times \mathcal{F}^s(q_1)
\]
are, respectively,
\[
\{(a,b,0,0,0) \mid a \in \mathbb{R}^{uu},\ b \in \mathbb{R}^u\} \quad \text{and} \quad \{(0,0,c,d,e) \mid c \in \mathbb{R},\ d \in \mathbb{R}^s,\ e \in \mathbb{R}^{ss}\}.
\]

Since \(f\) has the product form on   $\mathcal{F}^{ss}_{\delta/4}(q_1) \times (-1,1) \times \mathbb{T}^2$ and
\[
f\bigl(\mathcal{F}^{ss}_{\delta/4}(q_1) \times (-1,1) \times \mathbb{T}^2\bigr) \subset \mathcal{F}^{ss}_{\delta/4}(q_1) \times (-1,1) \times \mathbb{T}^2,
\]
it follows from the uniqueness of the splitting that
\[
\widetilde{E^{uu}}  \oplus E^{cu} = E^{uu} \oplus E^u \quad \text{and} \quad
E^{cs} = E^s \oplus T\mathbb{S} \oplus E^{ss}
\]
at every point of
\[
\mathcal{F}^{ss}_{\delta/4}(q_1) \times (-1,1) \times \mathbb{T}^2.
\]
Meanwhile, \(f\) is uniformly expanding along
 $E^{uu} \oplus E^{u}$
for all \(x \in \mathcal{F}^{ss}_{\delta/4}(q_1) \times (-1,1) \times \mathbb{T}^2\).

For any open disk \(D_0\) transverse to \(E^{cs}\) with
\[
D_0 \cap \bigl( \mathcal{F}^{ss}_{\delta/4}(q_1) \times (-1,1) \times \mathbb{T}^2 \bigr) \neq \varnothing,
\]
take a point
\[
\widetilde{x}\in D_0 \cap \bigl( \mathcal{F}^{ss}_{\delta/4}(q_1) \times (-1,1) \times \mathbb{T}^2 \bigr).
\]
By considering an appropriate subsequence, we may assume that
\[
f^{n_i}(\widetilde{x}) \longrightarrow (x_0,y_0,z_0)
\qquad \text{as } i\to+\infty .
\]
By the uniform expansion along
\[
\widetilde{E^{uu}} \oplus E^{cu}= E^{uu}\oplus E^{u} \quad \text{on} \quad\mathcal{F}^{ss}_{\delta/4}(q_1)\times (-1,1)\times \mathbb{T}^2
\]
together with the dominated splitting, 
the tangent spaces \( T_{f^{n_i}(\widetilde{x})}\big(f^{n_i}(D_0)\big) \) converge to
\[
\{(a,b,0,0,0)\mid a\in\mathbb{R}^{uu},; b\in\mathbb{R}^{u}\}.
\]
Moreover, a subdisk of \( f^{n_i}(D_0) \) (in the \(C^1\) topology)  converges to the disk \[
\mathcal{F}^{uu}_{\varepsilon}(x_0) \times \{y_0\} \times \mathcal{F}^{u}_{\varepsilon}(z_0).
\]

Since
\[
\mathcal{F}^{uu}_{\varepsilon}(x_0)\times\{y_0\}\times \mathcal{F}^{u}_{\varepsilon}(z_0)
\]
intersects transversely with
\[
\mathcal{F}^{ss}_{\delta/4}(q_1)\times (-1,1)\times \mathcal{F}^{s}(q_1)\subset W^{s}\bigl((q_1,0,q_1), f\bigr),
\]
it follows that the subdisk of \(f^{n_i}(D_0)\) which converges to
\[
\mathcal{F}^{uu}_{\varepsilon}(x_0) \times \{y_0\} \times \mathcal{F}^{u}_{\varepsilon}(z_0)
\]
also intersects transversely
\[
\mathcal{F}^{ss}_{\delta/4}(q_1) \times (-1,1) \times \mathcal{F}^{s}(q_1).
\]
Finally, by the invariance of
\[
W^{s}\bigl((q_1,0,q_1), f\bigr),
\]
the original disk \(D_0\) also intersects
\[
W^{s}\bigl((q_1,0,q_1), f\bigr)
\]
transversely.


For the disk \(D\) in the this Proposition~\ref{claim},  we may assume that \(D \subset M\) is an open disk; otherwise, we can replace \(D\) by its intersection with a sufficiently small open ball.  
There exists a point
\[
\widetilde{y}\in D\cap\bigl(\mathcal{F}^{ss}(q_1)\times(-1,1)\times\mathbb{T}^2\bigr).
\]
Hence, for some \(m>0\), the forward iterate \(f^{m}(\widetilde{y})\) lies inside
\[
\mathcal{F}^{ss}_{\delta/4}(q_1)\times(-1,1)\times\mathbb{T}^2.
\]
Since a diffeomorphism preserves transversality, the disk \( f^m(D) \), like \( D \), is also transverse to \( E^{cs} \). Therefore, \(f^{m}(D)\) can serve as the disk \(D_0\) used in the argument above.  By the invariance of
$
W^{s}\bigl((q_1,0,q_1),f\bigr)
$
again, we conclude that the disk \(D\) intersects \(W^{s}\bigl((q_1,0,q_1),f\bigr)\) transversely. A similar argument works for the case when
\[
D \cap \bigl( \mathcal{F}^{ss}(q_2) \times (0,2) \times \mathbb{T}^2 \bigr) \neq \varnothing.
\]
\end{proof}

\begin{lem}\label{skeleton22}
For any \( C^1 \)-disk \( D \) transverse to \( E^{cs} \), we have
\[
D \cap \bigl( \mathcal{F}^{ss}(q_1) \times (-1,1) \times \mathbb{T}^2 \bigr) \neq \varnothing \quad \text{or} \quad D \cap \bigl( \mathcal{F}^{ss}(q_2) \times (0,2) \times \mathbb{T}^2 \bigr) \neq \varnothing.
\]
\end{lem}
\begin{proof}
If \(D \subset \mathbb{T}^2 \times \{1\} \times \mathbb{T}^2\), then we only need to show that
\[
D \cap \bigl( \mathcal{F}^{ss}(q_1) \times (0,2) \times \mathbb{T}^2 \bigr) \neq \varnothing.
\]
If \(D \not\subset \mathbb{T}^2 \times \{1\} \times \mathbb{T}^2\), it suffices to show that
\[
D \cap \bigl( \mathcal{F}^{ss}(q_1) \times (-1,1) \times \mathbb{T}^2 \bigr) \neq \varnothing.
\]
Since the argument for the latter case also applies to the former, it suffices to consider only the case \(D \not\subset \mathbb{T}^2 \times \{1\} \times \mathbb{T}^2\).
Without loss of generality, we may assume that \(D \subset M\) is an open disk; otherwise, we can replace \(D\) by its intersection with a sufficiently small open ball.  


Due to the dominated splittings~\eqref{fenjiezicong} and~\eqref{dazicong}, the combination of Lemma~\ref{zhibiao} and Lemma~\ref{mutually} gives us
\begin{equation}\label{hh}
E^{cs} \subset E^u \oplus T\mathbb{S} \oplus E^s \oplus E^{ss}.
\end{equation}

There exists a point \((x_1, x_2, x_3) \in \operatorname{Int}(D) \cap \bigl( \mathbb{T}^2 \times (-1,1) \times \mathbb{T}^2 \bigr)\).
We can assume that \(x_1 \in \mathbb{T}^2\) admits an isometric square neighborhood \(U(x_1)\), which can be identified with
\[
\hat{\mathcal{F}}^{uu}_{\alpha}(x_1) \times \hat{\mathcal{F}}^{ss}_{\alpha}(x_1)( = U(x_1) ). 
\] for some \(\alpha>0\). We define the projection
\[
\pi^{uu,ss}_{uu} \colon
\hat{\mathcal{F}}^{uu}_{\alpha}(x_1) \times \hat{\mathcal{F}}^{ss}_{\alpha}(x_1)
\longrightarrow
\hat{\mathcal{F}}^{uu}_{\alpha}(x_1).
\]
(Without loss of generality, we identify \(x_1\) with the point \((0,0)\) in
\(\hat{\mathcal{F}}^{uu}_{\alpha}(x_1) \times \hat{\mathcal{F}}^{ss}_{\alpha}(x_1)\).)

Assume, on the contrary, that
\[
D\cap \bigl( \mathcal{F}^{ss}(q_1) \times (-1,1) \times \mathbb{T}^2 \bigr) = \emptyset.
\]
Then it follows that
\begin{equation}\label{non}
\pi_{1}^3(D) \subset \mathbb{T}^2 \setminus \mathcal{F}^{ss}(q_1).
\end{equation}
There exists a sufficiently small connected open subdisk
\(D_0 \subset D\) (with respect to the subspace topology) such that
\((x_1,x_2,x_3) \in D_0\) and
\[
\pi^3_1(D_0) \subset U(x_1).
\]
Combining the one-dimensional connectedness and relation~(\ref{non}), we conclude that
\[
\pi^{uu,ss}_{uu} \circ \pi^3_1 (D_0) = \{0 \}.
\]
It follows that
\[
D_0 \subset \mathcal{F}^{ss}_{\alpha}(x_1) \times (0,1) \times \mathbb{T}^2.
\]
Consequently, the tangent space of \(D_0\) 
\begin{equation}\label{hh2}
TD_0\subset E^{u} \oplus T\mathbb{S} \oplus E^{s} \oplus E^{ss}.
\end{equation}
The combination of relation~\eqref{hh2} and relation~\eqref{hh} implies that \( T D_0 + E^{cs} \subset E^{u} \oplus T\mathbb{S} \oplus E^{s} \oplus E^{ss} \). This directly contradicts the fact that \( D_0 \) is still transverse to \( E^{cs} \).
\end{proof}

\begin{thm}\label{skeleton} The set
$
S = \{(q_1,0,q_1), (q_2,1,q_2)\}$
is a skeleton for \(f\).   Moreover, the closure of 
\( W^{u}\bigl((q_1, 0, q_1), f\bigr) \) 
(resp. \( W^{u}\bigl((q_2, 1, q_2), f\bigr) \)) 
contains the points \( (p_1,0,p_1) \) and \( (q_2,0,q_2) \) 
(resp. \( (p_2,1,p_2) \) and \( (q_1,1,q_1) \)).
\end{thm}
\begin{proof}
Owing to the product form of \(f\) on the relevant region, it follows that
\[
 W^{s}\bigl((q_1,0,q_1), f\bigr)= \bigcup_{j \in \mathbb{Z}} f^j \bigl( \mathcal{F}^{ss}_{\delta/2}(q_1) \times (-1,1) \times \mathcal{F}^{s}(q_1) \bigr),  \]\[ W^{s}\bigl((q_2,1,q_2), f\bigr)= \bigcup_{j \in \mathbb{Z}} f^j \bigl( \mathcal{F}^{ss}_{\delta/2}(q_2) \times (0,2) \times \mathcal{F}^{s}(q_2) \bigr)   
\]
and
\[
W^{u}\bigl((q_i, i-1, q_i), f\bigr) = \bigcup_{j \in \mathbb{Z}} f^j \bigl( \mathcal{F}^{uu}_{\frac{\delta}{2}}(q_i) \times \{i-1\} \times \mathcal{F}^{u}(q_i) \bigr),\quad \text{for each $i\in\{1,2\}$}.
\]
Since \(f\) preserves \(\mathbb{T}^2 \times \{j\} \times \mathbb{T}^2\)  and
for each $j\in\{0,1\}$,  it follows that \(W^s((q_1,0,q_1), f)\) and \(W^u((q_2,1,q_2), f)\) are disjoint; hence, they clearly have no transverse intersections. Finally, combining Proposition~\ref{claim} and Lemma~\ref{skeleton22}, we conclude that \( S \) satisfies the first property of a skeleton.

For any point \( (\widetilde{x}, 0, \widetilde{z}) \), there exists a point
\( \widetilde{\widetilde{x}} \in \mathcal{F}^{uu}(q_1) \) that can be chosen
arbitrarily close to \( \widetilde{x} \).
For \( n \) sufficiently large, we have
\[
(A^{n_1})^{-n}(\widetilde{\widetilde{x}})
\in \mathcal{F}^{uu}_{\delta/2}(q_1).
\]
It then follows that
\[
\{(A^{n_1})^{-n}(\widetilde{\widetilde{x}})\}
\times \{0\} \times \mathcal{F}^{u}(q_i)
\subset W^{u}\bigl((q_1,0,q_1), f\bigr).
\]

Since \( f \) permutes the family
\[
\bigl\{
\{x\} \times \{y\} \times \mathcal{F}^{u}(z)
:\;
x \in \mathbb{T}^2,\;
y \in \mathbb{S},\;
z \in \mathbb{T}^2
\bigr\},
\]
it follows that
\[
f^{n}\bigl(
\{(A^{n_1})^{-n}(\widetilde{\widetilde{x}})\}
\times \{0\} \times \mathcal{F}^{u}(q_i)
\bigr)
\]
is of the form
\[
\{\widetilde{\widetilde{x}}\}
\times \{0\}
\times \mathcal{F}^{u}(\widetilde{\widetilde{z}})
\quad \text{for some }
\widetilde{\widetilde{z}} \in \mathbb{T}^2.
\]
Consequently, the set
\[
\{\widetilde{\widetilde{x}}\}
\times \{0\}
\times \mathcal{F}^{u}(\widetilde{\widetilde{z}})
\]
can be chosen arbitrarily close to the point
\( (\widetilde{x}, 0, \widetilde{z}) \).
By the invariance of
\( W^{u}\bigl((q_1,0,q_1), f\bigr) \),
we conclude that its closure coincides with
\[
\mathbb{T}^2 \times \{0\} \times \mathbb{T}^2,
\]
which in particular contains the point
\( (p_1,0,p_1) \) and \( (q_2,0,q_2) \).
The corresponding property for
\( W^{u}\bigl((q_2,1,q_2), f\bigr) \)
can be proved in an analogous manner.

\end{proof}

\section{Behavior of Skeletons  under Perturbations}\label{six}

\begin{thm}\label{reduction}
For any $C^1$-neighborhood $\U_f$ of $f$, there exists $C^\infty$ diffeomorphism $\widetilde{f}$ such that
skeleton of $\widetilde{f}$ is $\{(q_2,1,q_2)\}$.
\end{thm}
\begin{proof}
By the definition of the \(C^{1}\)-topology and the density of
\(\mathcal{F}^{uu}(q_1) \cap \mathcal{F}^{ss}(q_2)\) in \(\mathbb{T}^2\), there exists a sufficiently small \(\epsilon > 0\) such that the following properties hold.
\begin{itemize}
\item The \(C^{1}\)-neighborhood
\[
\V_f=\Bigl\{ \breve{f} :
\max_{x \in M} d\bigl(f(x), \breve{f}(x)\bigr) \le \epsilon,\
\max_{x \in M} \lVert D\breve{f}(x) - Df(x) \rVert \le \epsilon,\
\breve{f} \text{ is a diffeomorphism}
\Bigr\}
\]
is contained in \(\mathcal{U}_f\).
\item For every \(\breve{f} \in \mathcal{V}_f \), the diffeomorphism \( \breve{f} \) has a partially hyperbolic splitting of the same type as \( f \).
\item There exists a point \( r \in \mathcal{F}^{uu}(q_1) \cap \mathcal{F}^{ss}(q_2) \) such that
\[
\{ (A^{n_1})^{i}(r) : i \in \mathbb{Z} \setminus {0} \} \cap B_\epsilon(r) = \varnothing,
\]
and
\[
f\big|_{B(r,\epsilon) \times \mathbb{S} \times \mathbb{T}^2}
= A^{n_1} \times \mathrm{id} \times A^{n_0}.
\]
\end{itemize}

Choose \(\ell\) sufficiently large so that
\begin{equation}\label{definediffeo}
\frac{3\delta}{2\ell} \le \frac{\epsilon}{10000}
\quad \text{and} \quad
\max\{\frac{\sqrt{2}}{\ell^{2}}\cdot\max\{\psi'\}\cdot\max\{\psi\},\frac{1}{\ell^3}\} \le \frac{\epsilon}{10000}.
\end{equation}
We then define \(H_{\epsilon}\) by
\[
H_\epsilon : \mathbb{R}^{uu} \times \mathbb{R}^{ss} \times \mathbb{R} \times \mathbb{R}^{u} \times \mathbb{R}^{s}
\longrightarrow
\mathbb{R}^{uu} \times \mathbb{R}^{ss} \times \mathbb{R} \times \mathbb{R}^{u} \times \mathbb{R}^{s}
\]
by
\[
H_\epsilon(a,b,c,d,e) = \bigl(a, b, R(a,b,c), d, e\bigr),
\]
where\[
\begin{aligned}
R(a, b, c) &= \psi(\ell c) \psi\left( \ell \sqrt{a^2 + b^2} \right) \left( c + \frac{1}{\ell^3} \right) + \left( 1 - \psi(\ell c) \psi\left( \ell \sqrt{a^2 + b^2} \right) \right) c \\
&= c + \psi(\ell c) \psi\left( \ell \sqrt{a^2 + b^2} \right) \cdot \frac{1}{\ell^3}.
\end{aligned}
\]

We define \( \widetilde{f} \) in the following way, so that its lift \( \widehat{\widetilde{f}} \) satisfies
\[
\widehat{\widetilde{f}} := \hat{f} \circ H_\epsilon
\quad \text{on} \quad
\hat{\mathcal{F}}^{uu}_{\epsilon/10}(\hat{r})
\times \hat{\mathcal{F}}^{ss}_{\epsilon/10}(\hat{r})
\times \mathbb{R}
\times \hat{\mathcal{F}}^{u}_{\epsilon/10}(\hat{r})
\times \hat{\mathcal{F}}^{s}_{\epsilon/10}(\hat{r}).
\]
On the complement of this set, \( \widehat{\widetilde{f}} \) coincides with \( \hat{f} \).

After computation, we obtain:
\[
\frac{\partial R(a, b, c)}{\partial c} = 1 + \frac{1}{\ell^2} \psi'(\ell c) \cdot \psi\left( \ell \sqrt{a^2 + b^2} \right),
\]
\[
\frac{\partial R(a, b, c)}{\partial a} = \frac{1}{\ell^2} \psi(\ell c) \cdot \psi'\left( \ell \sqrt{a^2 + b^2} \right) \cdot \frac{a}{\sqrt{a^2 + b^2}},
\]
\[
\frac{\partial R(a, b, c)}{\partial b} = \frac{1}{\ell^2} \psi(\ell c) \cdot \psi'\left( \ell \sqrt{a^2 + b^2} \right) \cdot \frac{b}{\sqrt{a^2 + b^2}},
\]
and
\[
\lVert D\widetilde{f}(x) - Df(x) \rVert = \renewcommand{\arraystretch}{2}
\begin{pmatrix}
0 & 0 & 0 & 0 & 0 \\
0 & 0 & 0 & 0 & 0 \\
\frac{\partial R}{\partial a} & \frac{\partial R}{\partial b} & \frac{\partial R}{\partial a}-1 & 0 & 0 \\
0 & 0 & 0 & 0 & 0 \\
0 & 0 & 0 & 0 & 0
\end{pmatrix}.
\]
By the Cauchy inequality, we can conclude that:
\[
 \lVert D\widetilde{f}(x) - Df(x) \rVert = \frac{\sqrt{2}}{\ell^2} |\psi(\ell c) \cdot \psi'\left( \ell \sqrt{a^2 + b^2} \right)|
\]
Thus, by our setup in inequality~\eqref{definediffeo}, it follows that \( \widetilde{f} \in \mathcal{V}_f \).
It is straightforward to check that the following holds.
\begin{equation}\label{r1}
\text{The map } \widetilde{f} \text{ still permutes the family }
\bigl\{ \{x\} \times \{y\} \times \mathcal{F}^{u}(z) : x \in \mathbb{T}^2,\, y \in \mathbb{S},\, z \in \mathbb{T}^2 \bigr\}.
\end{equation}
\begin{equation}\label{r2}
\text{The point } (r,0) \text{ belongs to the intersection } 
W^s((q_2,1), \widetilde{g}) \cap W^u((q_1,0), \widetilde{g}),
\end{equation}
where \(\widetilde{g}\) satisfies
\[
 \pi^3_2 \circ \widetilde{f} = \widetilde{g} \circ \pi^3_2.
\]
Hence, we can choose \(\widetilde{\kappa}\) sufficiently large so that
\[
\widetilde{g}^{-\widetilde{\kappa}}(r,0) \in \mathcal{F}^{uu}_{\frac{\delta}{2}}((q_1,0)).
\]
Denote this point by \((r_{-\widetilde{\kappa}},0)\). Since \(\widetilde{f}\) still has the product form on \(B_{\frac{\delta}{2}}(p_1)\times(-1,1)\times\mathbb{T}^2\), it follows that
\[
\{r_{-\widetilde{\kappa}}\} \times \{0\} \times \mathcal{F}^u(q_1) \subset W^u((q_1,0,q_1),\widetilde{f}).
\]
Combining the previous observations~\eqref{r1}\eqref{r2}, there exists a sufficiently large \(n > 0\) and points \(\widetilde{p_1}, \widetilde{1}, \widetilde{w}\) such that
\[
\widetilde{f}^n \bigl( \{ r_{-\widetilde{\kappa}} \} \times \{0\} \times \mathcal{F}^u(q_1) \bigr) 
= \{ \widetilde{p_1} \} \times \{ \widetilde{1} \} \times \mathcal{F}^u(\widetilde{w}),
\]
with \(( \widetilde{p_1}, \widetilde{1}) \in \mathcal{F}^{ss}_{\frac{\delta}{2}}(q_2) \times (0,2)\). 
\begin{equation}
	\includegraphics[width=1.1\textwidth]{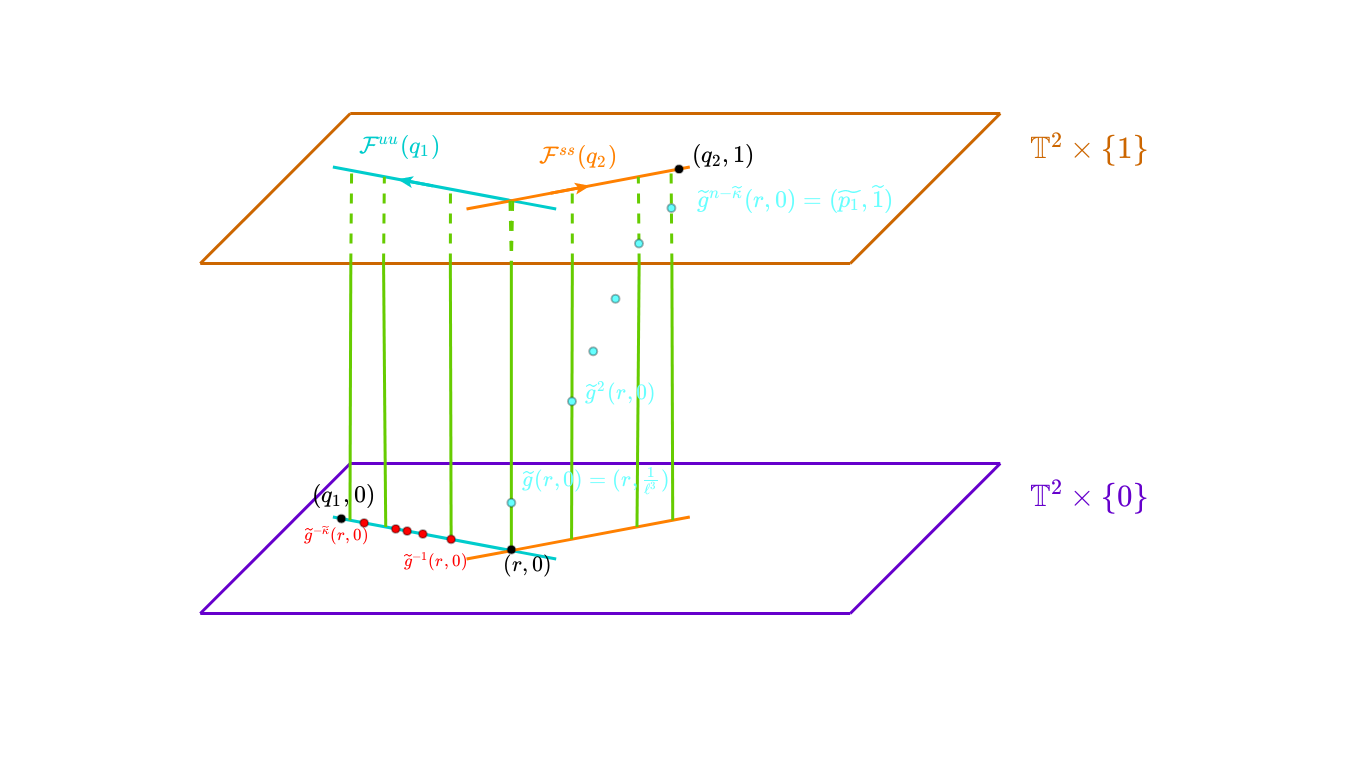}
\end{equation}
Since \(\widetilde{f}\) has the product form on \(\mathcal{F}^{ss}(q_2) \times (0,2) \times \mathbb{T}^2\), we have
\[
\mathcal{F}^{ss}(q_2) \times (0,2) \times \mathcal{F}^s(q_2) \subset W^s((q_2,1,q_2), \widetilde{f}).
\]
It then follows immediately that
\[
\bigl( \mathcal{F}^{ss}(q_2) \times (0,2) \times \mathcal{F}^s(q_2) \bigr) \cap
\bigl( \{ \widetilde{p_1} \} \times \{ \widetilde{1} \} \times \mathcal{F}^u(\widetilde{w}) \bigr) \neq \varnothing.
\]
It then follows that
 \[ 
 W^s((q_2,1,q_2), \widetilde{f}) \quad \text{and} \quad W^u((q_1,0,q_1), \widetilde{f}) 
 \] 
 intersect transversely at points in the set 
 \[ 
 \{\widetilde{p_1}\} \times \{\widetilde{1}\} \times \bigl( \mathcal{F}^u(\widetilde{w}) \cap \mathcal{F}^s(q_2) \bigr). 
 \]

The continuation \( S(\widetilde{f}) \) of the skeleton
\[
 S = \{(q_1,0,q_1), (q_2,1,q_2)\}
 \]
 remains
 \(\{(q_1,0,q_1), (q_2,1,q_2)\}\). We have \( \widetilde{f} \) on
\[
 B_{\delta/2}(q_1,\mathbb{T}^2)\times (-1,1) \times \mathbb{T}^2 \quad \text{and} \quad B_{\delta/2}(q_2,\mathbb{T}^2)\times (0,2) \times \mathbb{T}^2
\]
still has the same product form as \( f \).
Combining Proposition~\ref{claim} and Lemma~\ref{skeleton22}, we see that \( S(\widetilde{f}) \) satisfies the first property in the definition of a skeleton.
Moreover, since
\( W^s((q_2,1,q_2), \widetilde{f}) \) and \( W^u((q_1,0,q_1), \widetilde{f}) \) intersect transversely, the Inclination Lemma implies that any disk transverse to
 \( W^s((q_1,0,q_1), \widetilde{f}) \) must also intersect
 \( W^s((q_2,1,q_2), \widetilde{f}) \) transversely. Therefore,
 \(\{(q_2,1,q_2)\}\) is a skeleton of \( \widetilde{f} \).
\end{proof}

 Since Theorem~\ref{reduction} and the following Lemma~\ref{bubianshumu} are independent results, for simplicity, there are some repeated notations in the statement.

Denote by \( \mathbf{q_1} := (q_1, 0, q_1) \) and \( \mathbf{q_2} := (q_2, 1, q_2) \). Then, \( \mathbf{q_i}(\tilde{f}) \) is the continuation of \( \mathbf{q_i} \) under \( \tilde{f} \) for each \( i = 1, 2 \).

\begin{lem}\label{bubianshumu}
Let
\[
\mathcal{V} := \left\{ \breve{f} :  
\breve{f}(\mathbb{T}^2 \times \{i\} \times \mathbb{T}^2) = \mathbb{T}^2 \times \{i\} \times \mathbb{T}^2, \ i = 0, 1, \ \breve{f} \text{ is a diffeomorphism} \right\},
\]
then there exists a \( C^1 \)-neighborhood \( \mathcal{U}_f \) of \( f \) such that for any \( \tilde{f} \in \mathcal{U}_f \cap \mathcal{V}_f \),
\[
\overline{W^u(\mathbf{q_1}(\tilde{f}), \tilde{f})} = \mathbb{T}^2 \times \{0\} \times \mathbb{T}^2, \quad 
\overline{W^u(\mathbf{q_2}(\tilde{f}), \tilde{f})} = \mathbb{T}^2 \times \{1\} \times \mathbb{T}^2.
\]
\end{lem}
\begin{proof}
Since \[
f(x, 0, y) = (A^{n_1}(x), 0, A^{n_0}(y)) \quad \text{when} \quad (x, 0, y) \in B_{\delta/2}(q_1, \mathbb{T}^2) \times \{0\} \times \mathbb{T}^2,
\]
 it follows from the proof of \cite[Proposition~7.2]{CLM} that, up to shrinking \( \mathcal{U}_f \), \( W^u(\mathbf{q_1}(\tilde{f}), \tilde{f}) \) is dense in \( \mathbb{T}^2 \times \{0\} \times \mathbb{T}^2 \). A similar argument applies for \( \mathbf{q_2}(\tilde{f}) \).
\end{proof}



\section{Mostly Contracting Sub-Centers}\label{seven}

The purpose of this section is to establish the following theorem.

\begin{thm}\label{mostlycontracting}
The subbundle $\widetilde{T\mathbb{S}}$ is mostly contracting.
\end{thm}

\subsection{Projection of Ergodic Measures}

.\begin{lem}\label{pi}
Let \(\mu\) be any ergodic \(f\)-invariant measure and $(\pi^3_2)_*\mu
= \mu_{\mathbb{T}^2\times\mathbb{S}}$.
Then
\[
\int \log\Bigl|\det\bigl(Df|_{ \widetilde{T\mathbb{S}}(x)}\bigr)\Bigr| d\mu(x)=
\int \log\Bigl|\det\bigl(Dg|_{T\mathbb{S}(y)}\bigr)\Bigr|
d\mu_{\mathbb{T}^2\times\mathbb{S}}(y).
\]
\end{lem}
\begin{proof}
It is clear that \( \mu_{\mathbb{T}^2\times\mathbb{S}} \) is an ergodic \(g\)-invariant measure.  Since
\[
\pi^3_2\circ f = g \circ \pi^3_2\quad \text{and} \quad \widetilde{T\mathbb{S}} \subset \mathscr{C}_{\tau_3}(T\mathbb{S}, E^u),
\]
it follows from invariance that
\[
\det \bigl(D\pi^3_2|_{\widetilde{T\mathbb{S}}}\bigr)
\cdot \det \bigl(Df^n|_{\widetilde{T\mathbb{S}}}\bigr)
= \det \bigl(Dg^n|_{T\mathbb{S}}\bigr) \cdot \det \bigl(D\pi^3_2|_{\widetilde{T\mathbb{S}}}\bigr).
\]
Recall that \( \widetilde{T\mathbb{S}} \subset \mathscr{C}_{\tau_3}(T\mathbb{S}, E^u) \). Thus, we have the following inequality:
\[
\cos \theta_3 \leq \left| \det \left( D\pi^3_2 \big|_{\widetilde{T\mathbb{S}}} \right) \right| \leq 1,
\]
where \( \theta_3 = \arctan(\tau_3)\in[0,\frac{\pi}{2}) \). For better understanding, we present the following figure.
\begin{equation}\label{toushejiaodu}
	\includegraphics[width=1\textwidth]{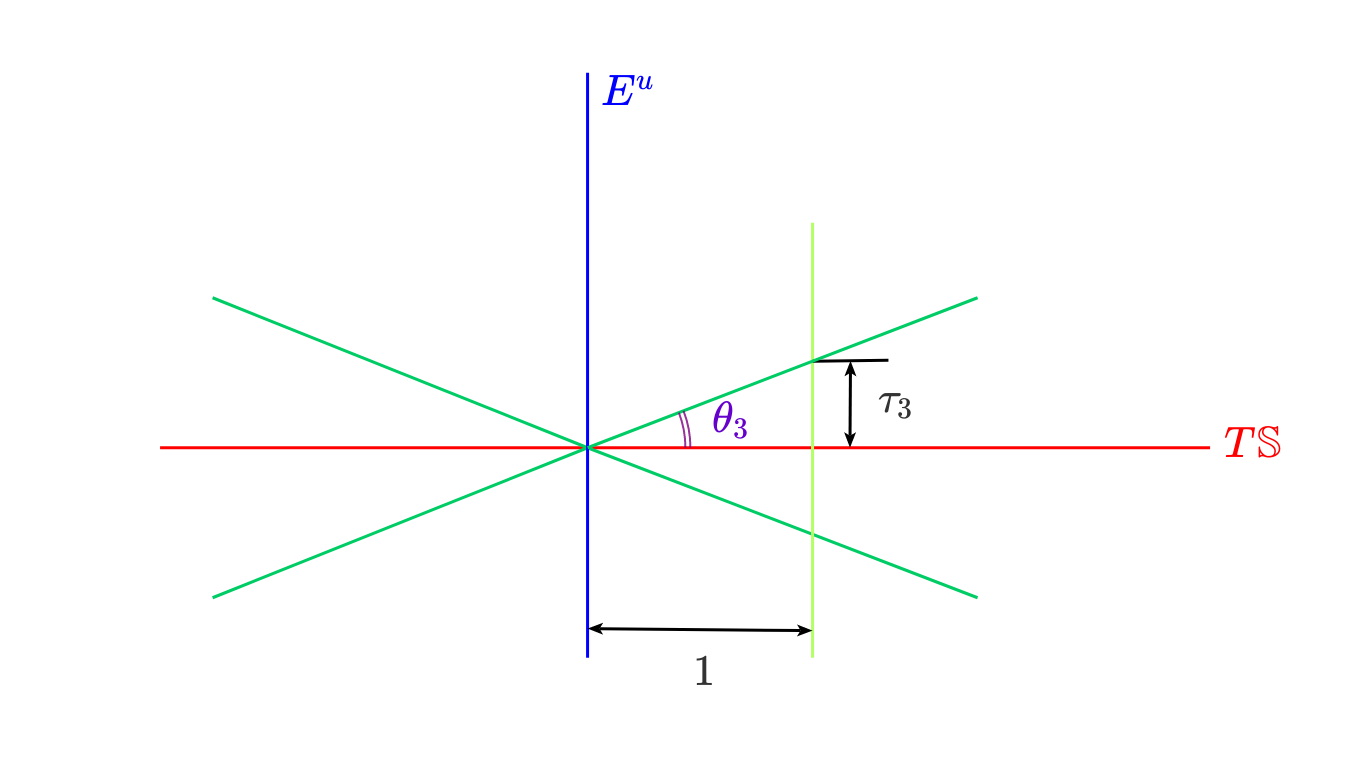}
\end{equation}
It follows that if \(y \in \mathbb{T}^2 \times \mathbb{S}\) satisfies
\[
\lim_{n \to \pm \infty} \frac{\log \bigl| \det \bigl(Dg^n|_{T\mathbb{S}}(y)\bigr) \bigr|}{n} \text{ exists,}
\]
then for any \(\tilde{x} \in \mathbb{T}^2\),
\[
\lim_{n \to \pm \infty} \frac{\log \bigl| \det \bigl(Df^n|_{\widetilde{T\mathbb{S}}}(y,\tilde{x})\bigr) \bigr|}{n}
= \lim_{n \to \pm \infty} \frac{\log \bigl| \det \bigl(Dg^n|_{T\mathbb{S}}(y)\bigr) \bigr|}{n}.
\]
In other words, in this case,
\[
\lim_{n \to \pm \infty} \frac{\log \bigl| \det \bigl(Df^n|_{\widetilde{T\mathbb{S}}}(y,\tilde{x})\bigr) \bigr|}{n}
\]
is independent of the choice of \(\tilde{x} \in \mathbb{T}^2\).

By the Birkhoff ergodic theorem, for this lemma, there exists a \(\mu_{\mathbb{T}^2 \times \mathbb{S}}\)-full measure set \(\Lambda \subset \mathbb{T}^2 \times \mathbb{S}\) such that for every \(y \in \Lambda\),
\[
\lim_{n \to \pm \infty} \frac{\log \bigl| \det \bigl(Dg^n|_{T\mathbb{S}}(y)\bigr) \bigr|}{n}
= \int \log \Bigl| \det \bigl(Dg|_{T\mathbb{S}}\bigr) \Bigr|  d\mu_{\mathbb{T}^2 \times \mathbb{S}}.
\]
It follows that \(\Lambda \times \mathbb{T}^2\) is a \(\mu\)-full measure set, and for every \((y, \tilde{x}) \in \Lambda \times \mathbb{T}^2\),
\[
\lim_{n \to \pm \infty} \frac{\log \bigl| \det \bigl(Df^n|_{\widetilde{T\mathbb{S}}}(y,\tilde{x})\bigr) \bigr|}{n}
= \int \log \Bigl| \det \bigl(Dg|_{T\mathbb{S}}\bigr) \Bigr|  d\mu_{\mathbb{T}^2 \times \mathbb{S}}.
\]
Hence, combining this with the Birkhoff ergodic theorem once again, we obtain
\[
\int \log\Bigl|\det\bigl(Df|_{ \widetilde{T\mathbb{S}}}\bigr)\Bigr| d\mu=
\int \log\Bigl|\det\bigl(Dg|_{T\mathbb{S}}\bigr)\Bigr|
d\mu_{\mathbb{T}^2\times\mathbb{S}}.
\]
\end{proof}

\subsection{On the Subbundles of the Splitting}

\begin{lem}\label{unstable}
$$\widetilde{E^{uu}_g} 
= D\pi^3_2
\bigl(\widetilde{E^{uu}}\bigr).$$
\end{lem}
\begin{proof}
By \eqref{exp1} and \eqref{exp2}, \(Df\) acting on  the bundle
\[
E^{uu} \oplus T\mathbb{S} \oplus E^{u} \oplus E^{s} \oplus E^{ss} \longrightarrow E^{uu} \oplus T\mathbb{S} \oplus E^{u} \oplus E^{s} \oplus E^{ss}
\]
takes the form
\begin{equation}\label{exp5}
\renewcommand{\arraystretch}{2.9}
\begin{pmatrix}
\sigma_1 & 0 & 0 & 0 & 0 \\
\frac{\partial Q}{\partial a} & \frac{\partial Q}{\partial c} & 0  & 0 & \frac{\partial Q}{\partial b} \\
\sigma_0\frac{\partial P}{\partial a}  & \sigma_0\frac{\partial P}{\partial c} & \sigma_0\frac{\partial P}{\partial d}  & \sigma_0\frac{\partial P}{\partial e} & \sigma_0\frac{\partial P}{\partial b} \\
0 & 0 & 0 & \frac{1}{\sigma_0} & 0 \\
0 & 0 & 0 & 0 & \frac{1}{\sigma_1}
\end{pmatrix}
\quad \text{or}\quad
\renewcommand{\arraystretch}{2.5}
\begin{pmatrix}
\sigma_1 & 0 & 0 & 0 & 0 \\
\frac{\partial Q}{\partial a} & \frac{\partial Q}{\partial c}& 0   & 0 & \frac{\partial Q}{\partial b} \\
0   & 0 & \sigma_0 & 0 & 0 \\
0 & 0 & 0 & \frac{1}{\sigma_0} & 0 \\
0 & 0 & 0 & 0 & \frac{1}{\sigma_1}
\end{pmatrix}
\end{equation}
We observe that \( E^{uu} \oplus T\mathbb{S} \oplus E^u \) is an invariant subbundle. Furthermore, 
\[
\mathscr{C}_{\tau_1}(E^{uu}, T\mathbb{S} \oplus E^u) \subset \mathscr{C}_{\tau_1}(E^{uu}, E^u \oplus T\mathbb{S} \oplus E^s),
\]
Since the latter is forward invariant from equation~\eqref{tau1}, \( \mathscr{C}_{\tau_1}(E^{uu}, T\mathbb{S} \oplus E^u) \) is also forward invariant. 
 Thus, from the forward invariance of \( \mathscr{C}_{\tau_1}(E^{uu}, T\mathbb{S} \oplus E^u) \), we can obtain the dominated splitting
\begin{equation}\label{domina}
E^{uu} \oplus T\mathbb{S} \oplus E^u
= \widetilde{\widetilde{E^{uu}}} \oplus_\succ \bigl(T\mathbb{S} \oplus E^u\bigr).
\end{equation}
Combining  Lemma~\ref{mutually} and the dominated splitting~\eqref{fenjiezicong}, we have:
\[
\left( E^u \oplus_\succ \widetilde{T\mathbb{S}} \right) \oplus_\succ \left( \widetilde{E^s} \oplus_\succ \widetilde{E^{ss}} \right).
\]
It then follows from relation~\eqref{lianggeES} and Lemma~\ref{mutually}  that
\begin{equation}\label{xin}
TM = \widetilde{\widetilde{E^{uu}}} \oplus_\succ \left( \left(T\mathbb{S} \oplus E^u\right) \oplus_\succ \left( \widetilde{E^s} \oplus_\succ \widetilde{E^{ss}} \right) \right).
\end{equation}
Since we have \( \dim \widetilde{E^{uu}} = \dim \widetilde{\widetilde{E^{uu}}} = 1 \) and the dominated splitting~\eqref{dazicong}, it follows from the second item of Lemma~\ref{zhibiao} that
\[
\widetilde{E^{uu}} =\widetilde{\widetilde{E^{uu}}} .
\]
Thus, to complete the proof of the lemma, it is enough to show that
\begin{equation}\label{deng}
\widetilde{E^{uu}_g} 
= D\pi^3_2
\bigl(\widetilde{\widetilde{E^{uu}}}\bigr).
\end{equation}
By the invariance of the subbundle, we have
\[
Dg \circ D\pi^3_2(\widetilde{\widetilde{E^{uu}}})=D\pi^3_2 \circ Df(\widetilde{\widetilde{E^{uu}}})=D\pi^3_2(\widetilde{\widetilde{E^{uu}}}).
\]
It follows that
\[
D\pi^3_2(\widetilde{\widetilde{E^{uu}}})
\]
is an invariant subbundle of \(Dg\). By the uniqueness of the dominated splitting, in order to establish Equation~\eqref{deng}, it suffices to show that
\[
D\pi^3_2(\widetilde{\widetilde{E^{uu}}}) \oplus_\succ T\mathbb{S}.
\]

Since relation~\eqref{domina} holds, we can choose a vector
\[
(1, \epsilon_1, \epsilon_2) \in \widetilde{\widetilde{E^{uu}}} \subset E^{uu} \oplus T\mathbb{S} \oplus E^u.
\]
Then
\[
\renewcommand{\arraystretch}{2.9}
\begin{pmatrix}
\sigma_1 & 0 & 0 \\
\frac{\partial Q}{\partial a} & \frac{\partial Q}{\partial c} & 0 \\
\sigma_0 \frac{\partial P}{\partial a} & \sigma_0 \frac{\partial P}{\partial c} & \sigma_0 \frac{\partial P}{\partial b}
\end{pmatrix}
\begin{pmatrix} 1 \\ \epsilon_1 \\ \epsilon_2 \end{pmatrix}=\begin{pmatrix}
\sigma_1 \\
\frac{\partial Q}{\partial a} + \frac{\partial Q}{\partial c} , \epsilon_1 \\
\sigma_0 \frac{\partial P}{\partial a} + \sigma_0 \frac{\partial P}{\partial c}, \epsilon_1 + \sigma_0 \frac{\partial P}{\partial b} , \epsilon_2
\end{pmatrix}.
\]
Combining the inclusion
\begin{equation}\label{gougu}
\widetilde{\widetilde{E^{uu}}} \subset \mathscr{C}_{\tau_1}(E^{uu}, T\mathbb{S} \oplus E^u),
\end{equation}
and 
\[
D\pi^3_2( \mathscr{C}_{\tau_1}(E^{uu}, T\mathbb{S} \oplus E^u)) =  \mathscr{C}_{\tau_1}(E^{uu}, T\mathbb{S}),
\]
we have
\[
D\pi^3_2\bigl(\widetilde{\widetilde{E^{uu}}}\bigr) \subset \mathscr{C}_{\tau_1}(E^{uu}, T\mathbb{S}),
\]
where \( \mathscr{C}_{\tau_1}(E^{uu}, T\mathbb{S}) \) refers to the cone in the system \( (\mathbb{T}^2 \times \mathbb{S}, g) \) (Therefore, it makes sense to write it this way). 
By the second item of Lemma~\ref{fact1}, we have
\[
\bigl\| Dg^n \big|_{D\pi^3_2(\widetilde{\widetilde{E^{uu}}})} \bigr\| \ge \frac{\sigma_1^n}{\sqrt{1+\tau_1^2}}.
\]
Recalling our construction, we have 
\[
\sigma_1  \gg\max \Biggl\{ \max\Bigl\{ \frac{\partial Q}{\partial c}(a,b,c) : a,b,c \in \mathbb{R} \Bigr\}, \Bigl( \min\Bigl\{ \frac{\partial Q}{\partial c}(a,b,c) : a,b,c \in \mathbb{R} \Bigr\} \Bigr)^{-1} \Biggr\}\ge\frac{\partial f_{\mathbb{T}^2 \times \mathbb{S}}(x, y)}{\partial y},\]
where $f_{\mathbb{T}^2 \times \mathbb{S}}(x, y)\in\mathbb{S}$.
It then follows that
\[
D\pi^3_2(\widetilde{\widetilde{E^{uu}}}) \oplus_\succ T\mathbb{S}.
\]
\end{proof}

Denote by \(\mathscr{F}^{uu}(f)\) and \(\mathscr{F}^{uu}(g)\) the strong unstable foliations of \(f\) and \(g\), respectively.
\begin{cor}\label{diffeomorphism}
We have
\[
\pi^3_2\bigl(\mathscr{F}^{uu}(f)\bigr) = \mathscr{F}^{uu}(g),
\]
and the map
\[
\pi^3_2 : \mathscr{F}^{uu}(x, f) \to \mathscr{F}^{uu}(\pi^3_2(x), g)
\]
is a diffeomorphism for every \(x\in\T^2\times\mathbb{S}\times\T^2\).
\end{cor}
\begin{proof}
The projection of  the strong unstable foliation of \(f\) onto \(\mathbb{T}^2 \times \mathbb{S}\) is given by
\[
\pi^3_2\bigl(\mathscr{F}^{uu}(f)\bigr)
:=
\Bigl\{
\pi^3_2\bigl(\mathscr{F}^{uu}((y,x),f)\bigr)
:\ (y,x)\in (\mathbb{T}^2 \times \mathbb{S}) \times \mathbb{T}^2
\Bigr\}.
\]

Combining Lemma~\ref{unstable} with the semiconjugacy relation
\[
\pi^3_2\circ f = g \circ \pi^3_2,
\]
we obtain that the projected foliation \(\pi^3_2\bigl(\mathscr{F}^{uu}(f)\bigr)\) is a \(g\)-invariant foliation everywhere tangent to the strong unstable subbundle \(\widetilde{E^{uu}_g}\).  By the uniqueness of the strong unstable foliation of \(g\), it follows that
\[
\pi^3_2\bigl(\mathscr{F}^{uu}(f)\bigr) = \mathscr{F}^{uu}(g).
\]
By Lemma~\ref{unstable} again, it is easy to see that \( \pi^3_2 \) is a diffeomorphism when restricted to each strong unstable leaf of $f$.
\end{proof}

\subsection{Factoring over \(A^{n_1}\) and Projected Measures of Gibbs \(u\)-States}

Recall that \( A^{n_1} \) is the fixed hyperbolic automorphism on \( \mathbb{T}^2 \), with its stable and unstable foliations being \( \mathcal{F}^{ss}(A^{n_1}) \) and \( \mathcal{F}^{uu}(A^{n_1}) \), respectively. 
 Consider a $C^{1+}$-diffeomorphism
\[
F : \mathbb{T}^2 \times N \to \mathbb{T}^2 \times N,\qquad
F(x,y) = \bigl(A^{n_1}(x), k(x,y)\bigr),
\]
where \(N\) is a smooth Riemannian manifold and \(F\) is a diffeomorphism.
Assume that 
\[
 \mathscr{C}_\alpha(E^{uu}, E^{ss}\oplus TN)
\]
is forward invariant for some $\alpha>0$. Then \(F\) admits a partially hyperbolic splitting
\[
T (\mathbb{T}^2 \times N) = E^{uu}_F \oplus_{\succ} (E^{ss} \oplus TN).
\]
It follows that \( F \) admits the strong unstable foliation \( \mathscr{F}^{uu}(F) \).
It is clear that \(E^{ss} \oplus TN\) is uniquely integrable, and the corresponding foliation is given by
\[
\mathscr{F}^{cs}(F) = \{\mathcal{F}^{ss}(x,A^{n_1}) \times N:x\in\T^2\}.
\]
Let
\[
\text{proj}: \mathbb{T}^2 \times N \longrightarrow \mathbb{T}^2
\]
be the projection defined by
\[
\text{proj}(t,m) = t,
\quad \text{for all } t \in \mathbb{T}^2, \ m \in N.
\]
Then we have
\begin{itemize}
\item $\text{proj} \circ F = A^{n_1} \circ \text{proj}$;
\item $\text{proj}(\mathcal{F}^{ss}(x,A^{n_1}) \times N) = \mathcal{F}^{ss}(x,A^{n_1})$ for every $x \in \T^2$;
\item The projection
\[
\text{proj} : \mathscr{F}^{uu}(x, F) \to \mathcal{F}^{uu}(\text{proj}(x), A^{n_1})
\]
is a homeomorphism for every \(x \in \mathbb{T}^2 \times N\).
\end{itemize}
It is then clear that \( F \) "factors over \( A^{n_1} \)" via \( \text{proj} \). The concept of factoring over Anosov, introduced by Ures, Viana, F. Yang, and J. Yang \cite{UVYY}, is used to study measures of maximal \( u \)-entropy.

Let $\mathcal{R} = \{ R_1, \ldots, R_m \}$ be a finite collection of closed sets forming a cover of $\mathbb{T}^2$, such that their interiors are pairwise disjoint. For each $i$, denote by $\F^{uu/ss}_i(x)$ the connected component of the intersection $\F^{uu/ss}(x, A^{n_1}) \cap R_i$ that contains the point $x$. 
We say that $\mathcal{R}$ is a  \textbf{ Markov partition} for $A^{n_1}$ if the following properties hold:
\begin{itemize}
\item Each $R_i$ is the closure of its interior, i.e., $R_i = \overline{\operatorname{Int}(R_i)}$.
\item For any $x, y \in R_i$, the intersection $\F^{uu}_i(x) \cap \F^{ss}_i(y)$ consists of exactly one point.
\item If $x \in \operatorname{Int}(R_i)$ and $A(x) \in \operatorname{Int}(R_j)$, then
  $$
  A^{n_1}(\F^{ss}_i(x)) \subset \F^{ss}_j(A^{n_1}(x)) \quad \text{and} \quad A^{n_1}(\F^{uu}_i(x)) \supset \F^{uu}_j(A^{n_1}(x)).
  $$
\end{itemize}
At this point, we define \( \mathcal{M} = \{ R_1 \times N, \ldots, R_m \times N \} \) as the Markov partition of \( F \).

For any $x \in R_i \times N$, the  \textbf{strong-unstable plaque} at $x$, denoted by $\mathscr{F}^{uu}_i(x,F)$, is defined as the connected component of
$$
\mathscr{F}^{uu}(x, F) \cap R_i \times N
$$
that contains $x$. It was explained by Ures, Viana, F. Yang, and J. Yang \cite{UVYY} that the strong unstable plaque defined in this way satisfies
\[
\text{proj}\bigl( \mathscr{F}^{uu}_i(x,F) \bigr) = \mathcal{F}^{uu}_i \bigl( \text{proj}(x) \bigr).
\]
The probability measure \( \nu^{uu}_{i,x,F} \) on \( \mathscr{F}^{uu}_i(x,F) \) is a \textbf{reference measure} if
\[
\text{proj}_* (\nu^{uu}_{i,x,F}) = \operatorname{vol}^u_{i, \text{proj}(x)},
\]
where \( \operatorname{vol}^u_{i, \text{proj}(x)} \) denotes the normalized Lebesgue measure on \( \mathcal{F}^{uu}_i \bigl( \text{proj}(x) \bigr)\).

Let \( \mu \) be an invariant measure of \( F \). For each \( i \), let \( \{\mu^{uu}_{i,x,F} : x \in R_i \times N\} \) denote the disintegration of the restriction \( \mu|_{R_i \times N} \) relative to the partition \( \{ \mathscr{F}^{uu}_i(x,F) : x \in R_i \times N \} \). Now we say \( \mu \) is a \textbf{\( c \)-Gibbs \( u \)-state} if for \( \mu \)-almost every \( x \), \[ \mu^{uu}_{i,x,F} = \nu^{uu}_{i,x,F} .\]
  The notation ``\( c \)'' in \( c \)-Gibbs \( u \)-states emphasizes that the associated reference measures possess locally constant Jacobians. This feature distinguishes \( c \)-Gibbs \( u \)-states from classical Gibbs \( u \)-states, which are defined using Lebesgue measures on strong unstable leaves.   Ures, Viana, F.~Yang, and J.~Yang \cite[Theorem~B]{UVYY} established an equivalence between measures of maximal \( u \)-entropy and \( c \)-Gibbs \( u \)-states.

\begin{cor}\label{gibbsc}
We have
\begin{itemize}
\item\cite[Proposition~4.1]{UVYY} For every \(x \in R_i \times N\) and every \(i \in \{1,\ldots,m\}\), every accumulation point of the sequence
\[
\mu_n = \frac{1}{n}\sum_{j=0}^{n-1} F^j_*( \nu^{uu}_{i,x,F})
\]
is a \(c\)-Gibbs \(u\)-state of $F$, where \(\nu^{uu}_{i,x,F}\) is the reference measure defined on the strong-unstable plaque \(\mathscr{F}^{uu}_{i}(x,F)\).
\item\cite[Corollary~3.7]{UVYY} For any \(c\)-Gibbs \(u\)-state \(\mu\) of \(F\), the pushforward measure \(\text{proj}_*(\mu)\) coincides with the Lebesgue measure on \(\mathbb{T}^2\).
\item \cite[Lemma~3.1,Corollary~3.7]{UVYY} An \(F\)-invariant probability measure \(\mu\) is a $c$-Gibbs \(u\)-state of \(F\) if and only if its \(u\)-entropy satisfies 
\[
h_\mu(F,\mathscr{F}^{uu}(F))=\log \sigma_{1},
\]
where \(\sigma_{1}\) denotes the largest  eigenvalue of \(A^{n_1}\). 
\end{itemize}
\end{cor}

 Next, we show
\begin{thm}\label{dengjia}
For any \(F\)-invariant measure \(\mu\), \(\mu\) is a Gibbs \(u\)-state if and only if it is a \(c\)-Gibbs \(u\)-state.
\end{thm}
\begin{proof}
By Equation~\eqref{LS}, which provides an equivalent characterization of Gibbs \(u\)-states, together with the third item  of Corollary~\ref{gibbsc}, it suffices to prove that for any \(F\)-invariant measure \(\mu\)
\begin{equation}\label{mubiao}
\int \log \bigl| \det DF|_{E^{uu}_F} \bigr| d\mu = \log \sigma_1.
\end{equation}

By Birkhoff's Ergodic Theorem, there exists a \(\mu\)-full measure set \(X \subset M\) such that for every \(x \in X\), the limit
\[
C(x) := \lim_{n \to \pm \infty} \frac{1}{n} \sum_{k=0}^{n-1} \log \bigl| \det DF|_{E^{uu}_F(F^k(x))} \bigr|
\]
exists. Moreover, \(C(x)\) is \(F\)-invariant (i.e., \(C(F(x)) = C(x)\)), and satisfies
\begin{equation}\label{ccc}
\int C(x)  d\mu(x) = \int \log \bigl| \det DF|_{E^{uu}_F} \bigr|  d\mu.
\end{equation}

Since
\[
 \frac{1}{n} \sum_{k=0}^{n-1} \log \bigl| \det DF|_{E^{uu}_F(F^k(w))} \bigr|
= \frac{1}{n} \log |\det \bigl( DF^n|_{E^{uu}_F(w)} \bigr) |=
 \frac{1}{n} \log \bigl\| DF^n|_{E^{uu}_F(w)} \bigr\|,
\]
where the last equality uses the fact that \(E^{uu}_F\) is one-dimensional.
Assume \(\dim N = m_0\). Choose any vector
\[
v \in \mathscr{C}_\alpha(E^{uu}, E^{ss} \oplus TN)
\]
such that its component along \(E^{uu}\) has unit length.  Notice that \(F(x,y) = \bigl(A^{n_1}(x), k(x,y)\bigr)\), so the component on the torus does not depend on the choice of \(y\).  Moreover, \( E^{ss} \oplus TN \) (\( = E^{ss} \times \{ 0 \} \oplus \{ 0 \} \times TN \)) is perpendicular to \( E^{uu} \) (\( = E^{uu} \times \{ 0 \} \)), i.e.,
\[
E^{ss} \oplus TN \perp E^{uu}.
\]
 By the forward invariance of the cone and the Pythagorean theorem, it follows that
\[
\sigma_1^n\le\|DF^n(v)\| \le \sqrt{1 + (m_0 + 1)\alpha^2}\cdot\sigma_1^n.
\]
It then follows that
\[
\lim_{n \to \pm\infty} \frac{1}{n} \log \bigl\| DF^n \big\|_{E^{uu}_F(v)} \bigr\|=\log\sigma_1.
\]
Hence, we obtain \( C(x) \equiv \log \sigma_1 \). Combining this with Equation~\eqref{ccc}, we obtain Equation~\eqref{mubiao}.
\end{proof}

\begin{pro}\label{gdebian}
For any ergodic Gibbs \(u\)-state \(\mu\) of \(f\), its projection onto \(\mathbb{T}^2 \times \mathbb{S}\) is given by $(\pi^3_2)_*(\mu)
=:\mu_{\mathbb{T}^2\times\mathbb{S}}$.
Then \(\mu_{\mathbb{T}^2\times\mathbb{S}}\) is an ergodic Gibbs \(u\)-state of \(g\).
\end{pro}
\begin{proof}
It is clear that \(f\) factors over \(A^{n_1}\) via \(\pi_1^3 =   \pi_1^2\circ \pi_2^3\), and that \(g \) factors over \(A^{n_1}\) via \(\pi_1^2\).    By Theorem~\ref{dengjia}, it is enough to show that \( \mu_{\mathbb{T}^2 \times \mathbb{S}} \) is also an ergodic \( c \)-Gibbs \( u \)-state of \( g \). Since ergodicity is straightforward to verify, we will focus on proving that \(\mu_{\mathbb{T}^2 \times \mathbb{S}}\) is a \(c\)-Gibbs \(u\)-state of $g$.

By Theorem~\ref{dengjia}, \(\mu\) is an ergodic \(c\)-Gibbs \(u\)-state of $f$.
By the definition of a \(c\)-Gibbs \(u\)-state, there exists a strong unstable plaque \(\mathscr{F}^{uu}_i(x,f)\) and a corresponding reference measure \(\nu^{uu}_{i,x,f}\) supported on this plaque such that
\[
\lim_{n \to +\infty} \frac{1}{n} \sum_{i=0}^{n-1} f^i_*(\nu^{uu}_{i,x,f}) = \mu.
\]
By continuity and the commutation relation
\[
\pi^3_2 \circ f^n = g^n \circ \pi^3_2,
\]
we obtain
\begin{equation}\label{dengshidengshi}
(\pi^3_2)_* \Biggl( \lim_{n \to +\infty} \frac{1}{n} \sum_{i=0}^{n-1} f^i_*(\nu^{uu}_{i,x,f}) \Biggr)
= (\pi^3_2)_*( \mu)
= \lim_{n \to +\infty} \frac{1}{n} \sum_{i=0}^{n-1} g^i_* \bigl( (\pi^3_2)_* (\nu^{uu}_{i,x,f}) \bigr).
\end{equation}
Moreover, the reference measures satisfy
\[
(\pi^3_1)_*(\nu^{uu}_{i,x,f}) = \operatorname{vol}^u_{i, \pi^3_1(x)},
\qquad \text{and } \pi^3_1 = \pi^2_1 \circ \pi^3_2.
\]
Hence,
\begin{equation}\label{tongpeig}
(\pi^2_1)_* \bigl( (\pi^3_2)_*( \nu^{uu}_{i,x,f}) \bigr) = \operatorname{vol}^u_{i, \pi^3_1(x)} \text{ defined on } \mathcal{F}^{uu}_i \bigl( \pi^2_1 \circ \pi^3_2(x), A^{n_1} \bigr).
\end{equation}
Recall that the strong unstable plaque at \( \pi^3_2(x) \) is \( \mathscr{F}^{uu}_i(\pi^3_2(x), g) \), which satisfies
\begin{equation}\label{tongpeig2}
\pi^2_1 \bigl( \mathscr{F}^{uu}_i(\pi^3_2(x), g) \bigr) = \mathcal{F}^{uu}_i \bigl( \pi^2_1 \circ \pi^3_2(x), A^{n_1} \bigr).
\end{equation}

Since \((\pi^3_2)_* (\nu^{uu}_{i,x,f})\) is defined on 
\[
\pi^3_2\bigl(\mathscr{F}^{uu}_i(x,f)\bigr),
\]
if \( \pi^3_2 \bigl( \mathscr{F}^{uu}_i(x,f) \bigr) = \mathscr{F}^{uu}_i(\pi^3_2(x), g) \), then combining \eqref{tongpeig} and \eqref{tongpeig2}, we obtain
\[
(\pi^3_2)_* (\nu^{uu}_{i,x,f}) = \nu^{uu}_{i,\pi^3_2(x),g},
\]
where \( \nu^{uu}_{i,\pi^3_2(x),g} \) is the reference measure on \( \mathscr{F}^{uu}_i(\pi^3_2(x), g) \).
At this point, combining equation~\eqref{dengshidengshi} and the first item of Lemma~\ref{gibbsc}, we have that
 \[ \mu_{\mathbb{T}^2 \times \mathbb{S}} = (\pi^3_2)_* (\mu) = \lim_{n \to +\infty} \frac{1}{n} \sum_{i=0}^{n-1} g^i_* \bigl(  \nu^{uu}_{i,\pi^3_2(x),g} \bigr)\] is a \( c \)-Gibbs \( u \)-state of \( g \).   For the reasons mentioned in this paragraph, we only need to focus on showing that 
\[
\pi^3_2 \bigl( \mathscr{F}^{uu}_i(x,f) \bigr) = \mathscr{F}^{uu}_i(\pi^3_2(x), g).
\]

By Corollary~\ref{diffeomorphism}, the map
\[
\pi^3_2 : \mathscr{F}^{uu}(x, f) \longrightarrow \mathscr{F}^{uu}(\pi^3_2(x), g)
\]
is a homeomorphism.  It follows that  $ \pi^3_2 \bigl( \mathscr{F}^{uu}_i(x,f) \bigr)\subset \mathscr{F}^{uu}(\pi^3_2(x), g)$. Moreover, the map
\[
\pi^2_1 : \mathscr{F}^{uu}(\pi^3_2(x), g) \longrightarrow \mathcal{F}^{uu}\bigl( \pi^2_1 \circ \pi^3_2(x), A^{n_1} \bigr)
\]
is also a homeomorphism.  Combine equation~\eqref{tongpeig2} and \(\pi^2_1 \circ \pi^3_2 \bigl( \mathscr{F}^{uu}_i(x,f) \bigr) = \mathcal{F}^{uu}_i\bigl( \pi^2_1 \circ \pi^3_2(x), A^{n_1} \bigr)\), we have
\[
\pi^3_2 \bigl( \mathscr{F}^{uu}_i(x,f) \bigr) = \mathscr{F}^{uu}_i(\pi^3_2(x), g).
\]
\end{proof}

\subsection{Proof of Theorem~\ref{mostlycontracting}}

\begin{proof}[Proof of Theorem~\ref{mostlycontracting}]
By Lemma~\ref{sat}, it suffices to show that for every ergodic Gibbs \(u\)-state \(\mu\) of \(f\),
\[
\int \log\Bigl|\det\bigl(Df|_{\widetilde{T\mathbb{S}}(x)}\bigr)\Bigr|d\mu(x) < 0.
\]
Assume from now on that \(\mu\) is an arbitrary ergodic Gibbs \(u\)-state of \(f\).

 By Proposition~\ref{gdebian}, \[\mu_{\mathbb{T}^2\times\mathbb{S}}=(\pi^3_2)_*\mu\] is an ergodic Gibbs \(u\)-state for \(g\).  By Lemma~\ref{mostcon}, 
\[
\int \log\Bigl|\det\bigl(Dg|_{T\mathbb{S}(y)}\bigr)\Bigr|
d\mu_{\mathbb{T}^2\times\mathbb{S}}(y)<0.\] By Lemma~\ref{pi},
\[
\int \log\Bigl|\det\bigl(Df|_{ \widetilde{T\mathbb{S}}(x)}\bigr)\Bigr| d\mu(x)=
\int \log\Bigl|\det\bigl(Dg|_{T\mathbb{S}(y)}\bigr)\Bigr|
d\mu_{\mathbb{T}^2\times\mathbb{S}}(y)<0,
\]which completes the proof.
\end{proof}

\section{Mostly Expanding Sub-Centers}\label{eight}

\begin{lem}\label{mostlyexpanding}
The subbundle $E^{cu}$ is mostly expanding.
\end{lem}
\begin{proof}
By Lemma~\ref{sat}, it suffices to show that for every ergodic Gibbs \(u\)-state \(\mu\) of \(f\),
\[
\int \log\Bigl|\det\bigl(Df|_{E^{cu}(x)}\bigr)\Bigr|d\mu(x) > 0.
\]

Assume from now on that \(\mu\) is an arbitrary ergodic Gibbs \(u\)-state of \(f\).  Combining Theorem~\ref{dengjia} with second item  of Corollary~\ref{gibbsc},  \[(\pi^3_1)_*\mu = \mathrm{Leb}_{\mathbb{T}^2}.\]
Recalling our construction, on the lifted space, we have
\[
\left|\det\left(D\hat{f}|_{E^{cu}(x)}\right)\right| \ge \frac{3}{4}
\]
on
\[
\hat{\mathcal{F}}^{uu}_{\frac{\beta}{2}}(\hat{p}_1) \times \hat{\mathcal{F}}^{ss}_{\frac{\beta}{2}}(\hat{p}_1) \times \mathbb{R} \times \hat{\mathcal{F}}^{u}_{5\delta}(\hat{p}_1) \times \hat{\mathcal{F}}^{s}_{5\delta}(\hat{p}_1),
\]
and
\[
\hat{\mathcal{F}}^{uu}_{\frac{\beta}{2}}(\hat{p}_2) \times \hat{\mathcal{F}}^{ss}_{\frac{\beta}{2}}(\hat{p}_2) \times \mathbb{R} \times \hat{\mathcal{F}}^{u}_{5\delta}(\hat{p}_2) \times \hat{\mathcal{F}}^{s}_{5\delta}(\hat{p}_2),
\]
while at all other points, 
\[
\left|\det\left(D\hat{f}|_{E^{cu}(x)}\right)\right| = \sigma_0.
\]
 It follows that
\[
\mu\left( \{x \in \mathbb{T}^2 \times \mathbb{S} \times \mathbb{T}^2 : \left|\det\left(Df|_{E^{cu}(x)}\right)\right| \le 1 \} \right) \le 2\beta^2.
\]
Then
\[
\int \log\Bigl|\det\bigl(Df|_{E^{cu}(x)}\bigr)\Bigr|  d\mu(x) \ge 2\beta^2 \log \frac{3}{4} + (1 - 2\beta^2) \log \sigma_0 > 0.
\]
\end{proof}

\section{Proof of the Main Theorem~\ref{main}}\label{nine}

\begin{proof}[Proof of Theorem~\ref{main}]
Since
\[
E^{cs}
=
\widetilde{T\mathbb{S}}
\oplus_{\succ}
\bigl(
\widetilde{E^{s}}
\oplus_{\succ}
\widetilde{E^{ss}}
\bigr),
\]
combining Lemma~\ref{daxiao} with Theorem~\ref{mostlycontracting}, we conclude that
\( E^{cs} \) is mostly contracting.  
Moreover, by Lemma~\ref{mostlyexpanding}, the bundle \( E^{cu} \) is mostly expanding.  
This completes the proof of the first property~(\ref{first}).

Combining Lemma~\ref{change} with Theorem~\ref{skeleton},  the support of the physical measure whose support is the closure of
\[
W^{u}\bigl((q_1,0,q_1), f\bigr)
\]
contains the points \( (p_1,0,p_1) \) and \( (q_2,0,q_2) \). At these points, we have
\[
\bigl\| Df|_{E^{cu}(p_1,0,p_1)} \bigr\| = \tfrac{3}{4}
\quad \text{and} \quad
\bigl\| Df|_{\widetilde{T\mathbb{S}}(q_2,0,q_2)} \bigr\| = \frac{3}{2}.
\]
Similarly,  the support of the physical measure  whose support is  the closure of \( W^{u}((q_2,1,q_2), f) \) have parallel properties. This completes the proof of property~\eqref{second}.

Combining Lemma~\ref{change} with Theorem~\ref{reduction} and Lemma~\ref{bubianshumu}, we complete the proof of property~\eqref{third}.
\end{proof}

\smallskip

\begin{sloppypar}
\flushleft{\bf Hangyue Zhang} \\
\small School of Mathematical Sciences,  Nanjing University, Nanjing, 210093, P.R. China\\
\textit{E-mail:} \texttt{zhanghangyue@nju.edu.cn}\\
\end{sloppypar}

\end{document}